\documentclass[12pt]{article}
\usepackage[english]{babel}
\usepackage{amsfonts,amsmath,amsxtra,amsthm,amssymb,latexsym}
\usepackage{mathrsfs}

\newtheorem{thm}{Theorem}[section]
 \newtheorem{cor}[thm]{Corollary}
 \newtheorem{lem}[thm]{Lemma}
 \newtheorem{prop}[thm]{Proposition}
 \theoremstyle{definition}
 \newtheorem{defn}[thm]{Definition}
 \theoremstyle{remark}
 \newtheorem{rem}[thm]{Remark}
 \newtheorem{ex}[thm]{Example}
 \numberwithin{equation}{section}

\newcommand{\limfunc}[1]{\mathop{\rm #1}}
\DeclareMathOperator{\im}{Im}

\DeclareMathOperator{\sgn}{sgn}

\DeclareMathOperator{\loc}{loc}

\DeclareMathOperator{\disc}{disc}
\DeclareMathOperator{\ess}{ess}
\DeclareMathOperator{\ac}{ac}
\DeclareMathOperator{\ran}{ran}
\DeclareMathOperator{\dom}{dom}
\DeclareMathOperator{\supp}{supp}

\def\R{\mathbb R}
\def\C{\mathbb C}
\def\N{\mathbb N}

\def\la{\lambda}
\def\H{\mathcal{H}}
\def\I{\mathcal{I}}
\def\A{\mathcal{A}}
\def\uph{\upharpoonright}
\def\ep{\varepsilon}
\def\wh{\widehat}
\newcommand{\F}{\mathcal{F}}
\newcommand{\Rs}{\mathcal{R}}
\newcommand{\D}{\mathcal{D}}
\newcommand{\Sl}{\mathrm{N}}
\newcommand{\pr}{\mathfrak{P}}
\def\LimInMed{\limfunc{l.i.m.}}
\def\MpmMm{\Phi}

\def\mur{\mu^r}
\def\mul{\mu^l}

\def\aexp{\mathfrak{a}}
\def\lexp{\mathfrak{l}}
\def\Dset{\mathfrak{D}}
\def\Sset{\mathcal{S}}

\def\c{\mathfrak c}

\def\L{\mathscr{L}}
\def\Iroot{\mathscr{S}}

\begin{document}

\title{A functional model, eigenvalues, and finite singular critical points for indefinite Sturm-Liouville operators.}

\author{I. M. Karabash}

\date{}

\maketitle

\begin{abstract}
Eigenvalues in the essential spectrum of a weighted Sturm-Liouville operator are studied under the assumption
that the weight function has one turning point. An abstract approach to the problem is given via a functional model
for indefinite Sturm-Liouville operators. Algebraic multiplicities of eigenvalues are obtained. Also,
operators with finite singular critical points are considered.
\end{abstract}


\quad

MSC-classes: 47E05, 34B24, 34B09 (Primary) 34L10, 47B50 (Secondary)

\quad


\quad

\section{Introduction}

Let $J$ be a signature operator 
in a complex Hilbert space $H$ (i.e., $ J=J^*=J^{-1}$).
Then $J=P_+ - P_- $ and $H=H_+ \oplus H_-$, where $P_\pm $ are the orthogonal
projections onto $H_\pm :=\ker (J \mp I)$.

Recall that a closed symmetric operator $S$ (in a Hilbert space $H$) is said to be
simple if there is no nontrivial reducing subspace in which $S$ is self-adjoint.

This paper is concerned mainly with $J$-self-adjoint operators $T$ such that
$T_{\min}:=T \cap T^*$ is a simple densely defined symmetric
operator in $H$ with the deficiency indices $n_+ (T_{\min}) = n_- (T_{\min})=2 $.
This class includes weighted Sturm-Liouville operators
\begin{gather}
A = \frac 1{r} \left( -\frac{d}{dx} p \frac{d}{dx} +q \right), \quad r, \frac{1}{p}, q \in L^1_{\loc} (a,b) , \quad -\infty \leq a <0 <b \leq + \infty, \label{e Aint}\\
xr(x) >0 \ \text{a.e. on} \ \R, \quad p >0 \ \text{a.e. on} \ \R , \quad q \ \text{is real-valued}, \label{e rint}
\end{gather}
equipped with separated self-adjoint boundary conditions at $a$ and $b$.
This statement is a consequence of the fact that the weight function $r$ has one turning point (i.e.,
the point where $r$ changes sign), see e.g. \cite{KarMMM07} and Section \ref{ss SLModel}. \eqref{e rint} fixes the turning point of $r$ at $0$, and therefore $A$ is $J$-self-adjoint in the weighted space $L^2 \left( (a,b) , \, |r(x)| dx \right)$ with the operator $J$
defined by $(Jf) (x):= (\sgn x)f(x)$.
Note that the case of one turning point of $r$ is principal for applications in kinetic theory 
(see \cite{BP83,B85,GvdMP87} and a short review in \cite[Section 1]{Kar08}).

The eigenvalue problem  for a regular indefinite Sturm-Liouville operator was
studied in a number of papers starting from Hilbert \cite{Hilbert} (see e.g. \cite{M83,B85,AM87,AM89,CL89,FlMon,Z05} and references therein). Till 2005, the spectral properties of singular  differential operators with an indefinite weight were studied mainly under the assumption of quasi J-nonnegativity, for $A$ this means that $\sigma(JA) \cap (-\infty,0)$ is finite, for the definition and basic results see \cite{CL89,FlMon}.
In last decades, big attention have been attracted by the problem of similarity of $A$
to a self-adjoint operator and the close problem of regularity of critical points (see a short review in \cite{KarKosM09}).

In this paper, the problem under consideration  is a  detailed description of the spectrum $\sigma (T)$ of the operator $T$,
of the set of eigenvalues (the point spectrum) $\sigma_p (T)$, and of algebraic and geometric multiplicities of eigenvalues.
In Section \ref{ss FM Js}, after some analysis of the more general case $n_+ (T_{\min}) = n_- (T_{\min}) < \infty$,
we assume the above conditions on $T_{\min}$ and construct a functional model of $T$ based on that of symmetric operators
\cite{DM95,GT00,MM03}. It occurs that, for the operator $A$, the main objects of this model are the spectral measures $d\Sigma_+$ and $d\Sigma_-$
of the classical Titchmarsh-Weyl m-coefficients associated with $A$ on $(a,0)$ and $(0,b)$
(
see Section \ref{ss SLModel} for details). In the abstract case, $d\Sigma_\pm$
are the spectral measures of two abstract Weyl functions $M_\pm$ (see \cite{DM91,DM95} for basic facts) naturally associated with $T$ and the signature operator $J$.

In Section \ref{s point s}, the model is used to find all eigenvalues of $T$ and their algebraic multiplicities in terms of $M_\pm$ and $d\Sigma_\pm$  (all geometric multiplicities are equal 1, the latter is obvious for the operator $A$). In turn, we obtain
 a description of the discrete and essential spectra and of the exceptional case when the resolvent set $\rho(T)$ is empty. For the operator $A$, these abstract results on the spectra of $T$ reduce the eigenvalue problem to
the problem of description of $M_\pm$ and $d\Sigma_\pm$ (or some of their properties) in terms of coefficients $p,q,r$. The latter problem is difficult,
but, fortunately, for some classes of coefficients is important for mathematical physics and is studied enough to get results on spectral properties of $A$
(see Sections \ref{ss nonempty} and \ref{s apl}).

 Non-emptiness  of $\rho(A)$ is nontrivial and essential for the spectral analysis of $A$
(see \cite{AM87,Ming04} and \cite[Problem 3.3]{KarKos08}). In Section \ref{ss nonempty}, the author generalizes slightly non-emptiness results noticed in \cite{KarKr05,KarMMM07,KarKos08}.

A part of this paper was obtained in the author's candidate thesis \cite{KarDis}, announced in the short communication \cite{KarKr05},
and used essentially in \cite{KarMMM07,KarKos08}. Some of these applications, as well as connections with \cite{KarKosM09,CL89,KarTr07} and with the similarity problem, are discussed in Section \ref{s apl}. 


Section \ref{s sing0} 
provides an alternative approach to the examples of $J$-self-adjoint Sturm-Liouville operators with the singular critical point $0$ given in
\cite[Sections 5 and 6]{KarKos08} and \cite[Section 5]{KarKosM09}. A class of operators with the singular critical point $0$
is constructed.
Relationships of the paper \cite{CL89}  with the example of \cite[Sections 6.1]{KarKos08} and with Theorem \ref{t s p} are discussed in Section \ref{s disc}.

The main advance of the method of the present paper is that it provides description of real eigenvalues and their algebraic multiplicities.
The answer is especially nontrivial and has a rich structure in the case of embedded eigenvalues. The interest to the latter problem is partially motivated by
the theory of kinetic equations of the Fokker-Plank (Kolmogorov) type (see references in Section \ref{ss other appl}).
Also we drop completely the conditions of quasi-$J$-positivity and definitizability.

The method of the paper is essentially based on the abstract approach to the theory of extensions of symmetric operators via boundary triplets, e.g. \cite{Koch75,DM91,DHMS00}. 
Some results on eigenvalues of non-self-adjoint extensions of symmetric operators were obtained in \cite{DM91,DM95,D99,Dth} with the use of this abstract approach. 
Relationships of these results with the results of the present paper are indicated in Remarks \ref{r sa} and \ref{r Sterms}.
There is a kindred approach to eigenvalue problem through characteristic functions, we refer the reader to the references in
\cite{DM99}. The characteristic function for the operator $A$ was calculated in \cite[Proposition 3.9]{KarMMM07}, but the analysis
of \cite{KarMMM07} shows that it is difficult to apply this method to eigenvalue problem for the operator $A$.
Connections with definitizability and local definitizability of $A$ and $T$ (see e.g. \cite{Lan82,J03} for basic facts and definitions) are given in
Remarks \ref{r defin} and \ref{r loc def}.


\textbf{Notation.}
Let $H$ and $\H$ be Hilbert spaces with the scalar products $ ( \cdot , \cdot )_{H}$ and $ ( \cdot , \cdot )_{\H}$, respectively.
The domain, kernel, and range of a (linear) operator  $S$ in $H$ is denoted
by $\dom (S)$, $\ker (S)$, and $\ran (S)$, respectively.
If $\Dset$ is a subset of $H$, then $S \Dset$ is the image of $\Dset$, $S \Dset := \{Sh \ : \ h \in \Dset \}$,
and $\overline \Dset$ is the closure of $\Dset$.

The discrete spectrum $\sigma_{\disc} (S)$ of $S$
is the set of isolated
eigenvalues of finite algebraic multiplicity. The essential spectrum is defined
by $\sigma_{\ess} (S):= \sigma (S) \setminus \sigma_{\disc} (S)$.
The continuous spectrum is understood in the sense
\[
\sigma_{c} (S) := \{ \lambda \in \C \setminus \sigma_{p} (S): \ran (S-\lambda) \neq \overline{\ran (S-\lambda)} = H \ \};
\]
$ \Rs_S (\lambda ):=\left( S-\lambda I\right)^{-1} $, $\lambda \in
\rho(S)$, is the resolvent of $S$.
Recall that an eigenvalue $\lambda$ of $S$ is called \emph{semi-simple} if $ \ker (S - \lambda)^2 = \ker (S - \lambda) $, and \emph{simple}
if it is semi-simple and $\dim \ker (S - \lambda) = 1$.
By $\Iroot_{\lambda} (S)$ we denote  the \emph{root subspace} (the algebraic eigensubspace) of $S$ corresponding to the point $\lambda$. That is,
$ \Iroot_{\lambda} (S)$ 
is the closed linear hull of the subspaces $\ker (S-\la )^k$, $k\in \N$.
 
If $S$ is a symmetric operator, $n_\pm (S)$ denote the deficiency indices of $S$ (see the Appendix).

The topological support $\supp d\Sigma$ of a Borel measure
$d\Sigma$ on $\R$ is the smallest closed set $\Sset$ such that
$d\Sigma (\R \setminus \Sset) = 0$; $d\Sigma (\{\la\})$ denotes the measure of point $\la$ (i.e., $d\Sigma (\{\la\}):= \Sigma (\la +0) - \Sigma (\la -0))$ if the measure $d\Sigma$ is determined by a function of bounded variation $\Sigma$.
We denote the indicator function of a set
$\Sset$ by $\chi_{\Sset} ( \cdot )$.
We write $f\in L^1_{\loc} (a,b)$  ($f\in AC_{\loc} (a,b)$) if the function $f$ is Lebesgue integrable
(absolutely continuous) on every closed bounded interval $[a',b'] \subset (a,b)$.

\section{The functional model for indefinite Sturm-Liouville operators with one turning point}

\subsection{Preliminaries; the functional model of a  symmetric operator}
\label{ss FM Sym}

Recall a functional model of symmetric operator
following \cite[Section 5.2]{DM95}, \cite[Section 7]{MM03} (a close version of a functional model can be found in \cite{GT00}).
In this paper, we need only the case of
deficiency indices $(1,1)$.

Let $\Sigma (t) $ be a nondecreasing scalar function satisfying the conditions
       \begin{multline}
\int_\R \frac{1}{1+t^2} d\Sigma (t) < \infty ,
\quad  \int_\R d\Sigma (t) = \infty \ , \\ 
\Sigma (t) = \frac 12 (\Sigma(t-0) + \Sigma (t+0)), \quad
\Sigma (0) = 0.      \label{e S assump}
         \end{multline}
The operator of multiplication $Q_{\Sigma} : f(t) \rightarrow tf(t)$ is self-adjoint  in
$L^2 (\R,d\Sigma (t))$. It is assumed that $Q_{\Sigma}$ is defined on its natural domain 
\[
\dom (Q_{\Sigma})= \{ f \in L^2 (\R,d\Sigma (t)) \, : \,
\int_\R |tf(t)|^2 d\Sigma (t) < \infty \}. 
\]
Consider the following restriction of $Q_{\Sigma}$:
\[
T_\Sigma = Q_\Sigma \upharpoonright \dom ( T_\Sigma), \qquad
\dom ( T_\Sigma) = \{ f \in \dom (Q_\Sigma) : \int_\R f(t)d\Sigma(t) =0\}.
\]
Then $ T_\Sigma$ is a simple densely defined symmetric operator in
$L^2 (\R, d\Sigma (t))$ with deficiency indices (1,1). The adjoint operator
$T_\Sigma^*$ has the form
\begin{multline}
\dom (T_\Sigma^*) = \{ f=f_Q + \c \, \frac{t}{t^2+1} \ : \
f_Q \in \dom (Q_\Sigma), \ \c \in \C \}, \\ 
T_\Sigma^* f = t f_Q - \c \, \frac 1{t^2 +1} ,  \label{e whT*}
\end{multline}
where the constant $\c$ is uniquely determined by the inclusion $f - \c t(t^2+1)^{-1} \in \dom (Q_\Sigma)$ due to the second condition in (\ref{e S assump}).

Let $C$ be a fixed real number.  Define linear mappings  $\Gamma_0^{\Sigma}$, $\Gamma_1^{\Sigma,C}$ from $\dom( T_\Sigma^*)$ onto $\C$ by
\begin{gather}
\Gamma_0^{\Sigma} f = \c , \qquad
\Gamma_1^{\Sigma,C} f = \c \,C +\int_\R f_Q (t) d\Sigma(t), \label{e GammaS} \\
\text{where } \qquad f= f_Q + \c \, \frac{t}{t^2+1} \in \dom( T_\Sigma^*),
\qquad f_Q \in \dom (Q_\Sigma), \qquad \c \in \C.
\notag
\end{gather}
Then $\{ \C, \Gamma_0^{\Sigma} , \Gamma_1^{\Sigma,C}\}$ is a boundary triplet for $T_\Sigma^*$ (see \cite[Proposition 5.2 (3)]{DM95}, basic facts on boundary triplets and abstract Weyl functions are given in the Appendix). The function
\begin{gather} \label{e Msc}
 M_{\Sigma,C} (\lambda) := C+\int_\R \left( \frac 1{t-\lambda} -
\frac t{1+t^2} \right) d\Sigma(t) , \quad \lambda
\in \C \setminus \supp d\Sigma,
\end{gather}
is the corresponding Weyl function of $  T_\Sigma $.

Another way to describe the operator
$ T_\Sigma^*$ is the following (see \cite{DM95}).
Note that the domain 
$ \dom \left( T_\Sigma^* \right) $
consists of the functions $f \in L^2 (\R, d\Sigma(t))$
such that for some constant $\c \in \C $ the function
$ \widetilde f  (t) := tf (t) - \c $ belongs to
$ L^2 (\R, d\Sigma (t))$.
It follows from \eqref{e S assump} that the constant $\c$
is uniquely determined and coincides with the constant $\c$ introduced in (\ref{e whT*}). Therefore,
\begin{gather}
\c = \Gamma_0^{\Sigma} f  \quad \text{and} \quad 
 T_\Sigma^* f =  \widetilde f \label{e hatA*} .
\end{gather}

%

\subsection{The functional model for J-self-adjoint extensions
of symmetric operators}
\label{ss FM Js}

Let $J$ be a signature operator in a Hilbert space $H$, i.e.,
$\ J=J^*=J^{-1}$. Then $J=P_+ - P_- $ and $H=H_+ \oplus H_-$, where $P_\pm $ are the orthogonal
projections onto $H_\pm :=\ker (J \mp I)$.

Let $T$ be a $J$-self-adjoint operator in $H$, i.e.,
the operator $B=JT$ is self-adjoint.
The domains of $T$ and $B$ coincide, we denote them by
\[
 \D := \dom(T) \ (= \dom (B)).
\]
Put
\[
T_{\min} := T \cap T^*, \qquad \D_{\min} := \dom(T_{\min}).
\]
By the definition, the operator $T_{\min}$ is a symmetric operator
and so is
\begin{equation} \label{e defBmin}
B_{\min} := B \upharpoonright \D_{\min} \quad =J T_{\min} .
\end{equation}


Let $\Sigma_+$ and $\Sigma_-$ be nondecreasing scalar functions satisfying
\eqref{e S assump}. Let $C_+$ and $C_-$ be real constants.
Consider the operator
$\widehat A = \widehat A \; \{ \Sigma_+ , C_+ , \Sigma_- , C_- \}$
in $L^2 (\R, d\Sigma_+ ) \oplus L^2 (\R, d\Sigma_-)$ defined by
\begin{equation}
 \widehat A \ \{ \Sigma_+ , C_+ , \Sigma_- , C_- \} =
T_{\Sigma_+}^* \oplus T_{\Sigma_-}^* \upharpoonright
\dom (\widehat A)  ,  \label{e A hat}
\end{equation}
\begin{multline*}
\dom (\widehat A) = \{ \ f=f_+ + f_- \ : \ f_\pm \in \dom (T_{\Sigma_\pm}^*) , \\
  \Gamma_0^{\Sigma_+} f_+ = \Gamma_0^{\Sigma_-} f_- , \
\Gamma_1^{\Sigma_+, C_+} f_+ = \Gamma_1^{\Sigma_-, C_-}  f_- \ \} ,
\end{multline*}
where $ T_{\Sigma_\pm}^*$ are the operators defined in Subsection \ref{ss FM Sym}.

One of main results of this paper is the following theorem.

\begin{thm} \label{p CModel}
Let $J$ be a signature operator in a separable Hilbert space $H$ and let $T$ be a $J$-self-adjoint operator such that
$T_{\min}:=T \cap T^*$ is a simple densely defined symmetric
operator in $H$ with deficiency indices (2,2).
Then there exist nondecreasing scalar
functions $\Sigma_+$, $\Sigma_-$ satisfying \eqref{e S assump} and real constants
$C_+$, $C_-$ such that $T$ is unitarily equivalent to
the operator $\widehat A \ \{ \Sigma_+ , C_+ , \Sigma_- , C_- \}$.
\end{thm}

First, we prove several propositions that describe the structure of
$T$ as an extension of the symmetric operator $T_{\min}$, and then prove Theorem \ref{p CModel} at the end of this subsection.

\begin{prop} \label{p Am pm}
Let $T$ be a J-self-adjoint operator.
Let the operators $T_{\min}^\pm$ be defined by
\begin{equation} \label{e Dmin}
T_{\min}^\pm := T \upharpoonright \D_{\min}^\pm , \quad
\D_{\min}^\pm = \dom (T_{\min}^\pm) := \D_{\min} \cap H_\pm.
\end{equation}
 Then:
\begin{description}
\item[(i)]  $T_{\min}^\pm$ is a symmetric operator in the Hilbert space $H_\pm$ and
\begin{gather} 
T_{\min} = T_{\min}^+ \oplus T_{\min}^- , \quad
B_{\min} = B_{\min}^+ \oplus B_{\min}^-, \quad \text{where} \quad
B_{\min}^\pm := \pm T_{\min}^\pm . \label{e Bmin=B+B}
\end{gather}
\item[(ii)] If any of the following two conditions  \\
(a) \quad $\rho (T) \neq \emptyset$ , \\
(b) \quad $n_+ (T_{\min}) = n_- (T_{\min})$, \\
is satisfied, then
$n_+ (T_{\min}^+) = n_- (T_{\min}^+)$ and
$n_+ (T_{\min}^-) = n_- (T_{\min}^-)$.\\
In particular, (a) implies (b).
\end{description}
\end{prop}

\begin{proof}
\textbf{(i)} \
Since $B=B^*$ and
$ \D = \dom (T) = \dom(B)$, we have $T^*=BJ$ and
$\D_{\min} = \{ f \in \D \cap J \D : JB f = BJ f \}$.
So if $f \in \D_{\min}$ and $g=Jf $, then 
$g \in \D \cap J \D$ and 
\[
JB g = JBJf= JJBf = Bf = B Jg .
\]
This implies $J \D_{\min} \subset \D_{\min}$ (and 
in turn  $J \D_{\min} = \D_{\min}$ since $J$
is a unitary operator).
Hence, for $ f \in \D_{\min}$ we have
$ P_+ f + P_- f \in \D_{\min}$ and
$ P_+ f - P_- f \in \D_{\min}$.
So
$ P_\pm f \in \D_{\min}$ and
$\D_{\min} \subset ( \D_{\min} \cap  H_+ ) \oplus (\D_{\min} \cap H_-) $. The inverse inclusion is obvious, and we see that
\[
\D_{\min} = ( \D_{\min} \cap  H_+ ) \oplus 
(\D_{\min} \cap H_-) .
\]

Now note that
$T_{\min} (\D_{\min} \cap H_\pm )  \subset  H_\pm $ .
Indeed, let $f_\pm \in \D_{\min} \cap H_\pm $. Since
$J f_\pm = \pm f_\pm$ and $T f_\pm= T^* f_\pm$, we see that
\begin{equation} \label{e Bfpm}
 JB f_\pm = BJ f_\pm = \pm B f_\pm .
\end{equation}
Note that $g \in H_\pm$ is equivalent to 
$J g = \pm g$. So (\ref{e Bfpm}) implies $B f_\pm \in H_\pm$, and therefore the vector $T_{\min} f_\pm = T f_\pm = JBf_\pm $ belongs to $H_\pm$.

The first part of \eqref{e Bmin=B+B} is proved.
Since $T_{\min}$ is a symmetric operator in $H$, the operators $T_{\min}^\pm$ are symmetric too.
Finally, the second part of \eqref{e Bmin=B+B} follows from \eqref{e defBmin} and \eqref{e Bfpm}.

\textbf{(ii)}
Since $B=B^*$, it easy to see that
\begin{gather}
n_+ (B_{\min}^+ ) + n_+ (B_{\min}^-) = n_+ (B_{\min}) = n_- (B_{\min})
= n_- (B_{\min}^+ ) + n_- (B_{\min}^-) =: m  . \label{e nB}
\end{gather}
The equalities \ $n_\pm (T_{\min}^+) = n_\pm (B_{\min}^+)$ \ and  \ $\ n_\pm (T_{\min}^-) = n_\mp (B_{\min}^-)$ \ imply \begin{gather} \label{e nT}
n_\pm (T_{\min}) = n_\pm (T_{\min}^+ ) + n_\pm (T_{\min}^-) = n_\pm (B_{\min}^+ ) + n_\mp (B_{\min}^-) .
\end{gather}

It follows from \eqref{e nB} and \eqref{e nT} that $n_+ (T_{\min}) >m$ yields $n_- (T_{\min}) < m$. In this case, $\C_- \subset \sigma_p (T)$
and $ H \neq \overline{(T-\lambda I) \dom (T)}$ for $\la \in \C_+$. Hence, $\rho (T) = \emptyset $.
The case $n_+ (T_{\min}) <m$, $n_- (T_{\min}) > m$ is similar.
Thus, if $\rho (T) \neq \emptyset$ or $n_+ (T_{\min}) = n_- (T_{\min})$, then
\[
n_+ (B_{\min}^+ ) + n_- (B_{\min}^-) = n_- (B_{\min}^+ ) + n_+ (B_{\min}^-) = m.
\]

Using \eqref{e nB}, we get \ $n_+ (B_{\min}^\pm) = n_- (B_{\min}^\pm)$ \ and, therefore,
 \ $n_+ (T_{\min}^\pm) = n_- (T_{\min}^\pm)$.
\end{proof}

Assume now that the operator $T_{\min}$ is densely defined in $H$. Put, for convenience' sake,
\[
T_{\max}^\pm = (T_{\min}^\pm)^*, \qquad B_{\max}^\pm = (B_{\min}^\pm)^*.
\]
 Clearly,
\begin{equation} \label{e TmaxBmax}
T_{\max}^\pm = \pm B_{\max}^\pm \quad \text{and} \quad \dom( T_{\max}^\pm) = \dom( B_{\max}^\pm) =:\D_{\max}^\pm.
\end{equation}

\begin{prop} \label{p PpmIsom}
Let $T$ be a J-self-adjoint operator.
Assume that $T_{\min}$ is densely defined in $H$ and
\[
n_+ (T_{\min}^+) = n_- (T_{\min}^+) =: N^+ < \infty, \qquad
n_+ (T_{\min}^-) = n_- (T_{\min}^-) =: N^- <\infty.
\]
Then:
\begin{description}
\item[(i)]
$ n_+ (T_{\min}^+) = n_- (T_{\min}^+) =
n_+ (T_{\min}^-) = n_- (T_{\min}^-) $, \ \text{that is}, \ $N^+=N^-=:N$ ;
\item[(ii)] the mappings $\pr_\pm := P_\pm \uph \D /\D_{\min}$ are well-defined and are linear isomorphisms from the quotient space
$\D/\D_{\min}$ onto the quotient space $\D_{\max}^{\pm}/\D_{\min}^{\pm}$.
\end{description}
\end{prop}

\begin{proof}
Note that
\begin{equation} \label{e defDmax}
\D_{\max} := \D_{\max}^+ \oplus \D_{\max}^-
\end{equation}
 is a domain of both the operators $ T_{\max}:=T_{\min}^*$ and
$ B_{\max} := B_{\min}^*$. Since 
\[ \D=\dom(B) \subset \dom (B_{\max}) \quad 
\text{and} \quad P_\pm \D_{\min} = \D_{\min}^\pm, 
\]
we see that
$\pr_\pm : \D /\D_{\min} \to  \D_{\max}^{\pm}/\D_{\min}^{\pm}$
are well-defined linear mappings.

Let us show that
\begin{equation} \label{e PpmInj}
\text{the mappings}\quad \pr_\pm \quad \text{are injective.}
\end{equation}
Indeed, if $\ker \pr_+ \neq \{0\}$, then there exists
$h \in \D $  such that $ h \not \in \D_{\min}$ and $P_+ h \in \D_{\min}^+ $. Recall that $\D_{\min}^+ \subset \D$, so $ P_- h = h - P_+ h \in \D \cap H_- = \D_{\min}^-$. By the first equality in \eqref{e Bmin=B+B}, $\D_{\min} = \D_{\min}^+ \oplus \D_{\min}^-$ and this implies that $h = P_+ h + P_- h$ belongs to $\D_{\min}$, a contradiction.

Since $2N^\pm = \dim (\D_{\max}^\pm / \D_{\min}^\pm)$, it follows from \eqref{e PpmInj} that
\begin{equation} \label{e Npm>m}
2N^+ \geq m ,\quad 2N^- \geq m, \quad \text{where} \quad m:=\dim (\D/\D_{\min})
\end{equation}
(this definition of $m$ coincides this that of \eqref{e nB}). Indeed,
$B$ is a self-adjoint extension of $B_{\min}$, therefore,
\[
\dim \left(\D/\D_{\min} \right) = n_+ (B_{\min}^+) + n_+ (B_{\min}^-) = n_- (B_{\min}^+) + n_- (B_{\min}^-).
\]
We see that $\ m= N^+ + N^-$.
From this and \eqref{e Npm>m}, we get $N^+=N^-=m/2$.
Thus, statement \textbf{(i)} holds true.
Besides, taking \eqref{e PpmInj} and $N^\pm < \infty$ into account, one obtains that
$\pr_\pm$ are surjective. This complete the proof of \textbf{(ii)}.
\end{proof}

Recall that existence of a boundary triplet for $S^*$,
where $S$ is a symmetric operator in a separable Hilbert space $H$, is equivalent to $n_+ (S) = n_- (S) $ (see \cite{Koch75,DM91}).

\begin{thm}[cf. Theorem 6.4 of \cite{DHMS00}] \label{l G+=G-}
Let $T$ be a J-self-adjoint operator.
Assume that $T_{\min}$ is densely defined in $H$ and
$n_+ (T_{\min}) = n_- (T_{\min}) =: m < \infty$.
Then:
\begin{description}
\item[(i)] $m$ is an even number and
\[
 n_+ (T_{\min}^+) = n_- (T_{\min}^+) =
n_+ (T_{\min}^-) = n_- (T_{\min}^-) = m/2.
\]
\item[(ii)]
Let $\{ \C^{m/2}, \Gamma_0^+ , \Gamma_1^+\}$ be a boundary triple for $T_{\max}^+ $.
Then there exist a boundary triple \\
$\{ \C^{m/2}, \Gamma_0^- , \Gamma_1^- \}$ for $T_{\max}^- $ such that
\begin{gather} \label{e D}
\D = \{h \in \D_{\max}: \quad
\Gamma_0^+ P_+ h =\Gamma_0^- P_- h, \quad \Gamma_1^+ P_+ h =\Gamma_1^- P_- h \}
\end{gather}
(note that $P_\pm h \in \D_{\max}^\pm $ due to \eqref{e defDmax}).
\end{description}
\end{thm}

Theorem \ref{l G+=G-} shows that the operator $T$ admits the representation
\begin{gather} \label{e T=}
T=T_{\max}^+ \oplus T_{\max}^- \uph \D, \quad \text{and its domain} \
\D \ \text{has the representation \eqref{e D}}.
\end{gather}

\begin{proof}
\textbf{(i)} follows from Proposition \ref{p PpmIsom} (i).

\textbf{(ii)} Let $\{ \C^{m/2}, \Gamma_0^+ , \Gamma_1^+\}$ be a boundary triple for the operator $T_{\max}^+ $
(actually, statement \textbf{(i)} implies that such a boundary triple exists,
for the case when the space $H$ is separable see e.g. \cite{Koch75}).
It follows from Definition \ref{dI.3}, that $\Gamma_0^+ \D_{\min}^+ = \Gamma_1^+ \D_{\min}^+ = \{ 0\}$. So
one can consider the mappings $\Gamma^+ : h_+ \to \{ \Gamma_0^+ h_+, \Gamma_1^+ h_+\}$ as a linear isomorphism from $\D_{\max}^+ /\D_{\min}^+$ onto $\C^{m/2} \oplus \C^{m/2}$.
Introducing the mappings
\begin{gather} \label{e G-=G+}
\Gamma_j^- := \Gamma_j^+ \pr_+ \pr_-^{-1} , \quad j=0,1,
\end{gather}
one can get from Proposition \ref{p PpmIsom} (ii) the fact that
$\Gamma^- : h_- \to \{ \Gamma_0^- h_-, \Gamma_1^- h_-\}$
is a linear isomorphism from $\D_{\max}^- /\D_{\min}^-$ onto $\C^{m/2} \oplus \C^{m/2}$. Putting $\Gamma_0^- h_- =\Gamma_1^- h_- = 0$
for all $h_- \in \D_{\min}^-$, we get natural linear extensions of  $\Gamma_0^-$,  $\Gamma_1^-$, and $\Gamma^-$ on $\D_{\max}^-$.

Let $\mathfrak{h} \in \D_{\max}/\D_{\min}$ and $\mathfrak{h_\pm} = \widetilde P_\pm \mathfrak{h}$, where $\widetilde P_\pm $ are mappings from $\D_{\max}/\D_{\min}$ to $\D_{\max}^\pm/\D_{\min}^\pm$ induced by $P_\pm$.
By Proposition \ref{p PpmIsom}~(ii), $\mathfrak{h} \in \D /\D_{\min}$
if and only if $\mathfrak{h}_+ = \pr_+ \pr_-^{-1} \mathfrak{h}_-$.
From this and \eqref{e G-=G+}, one can obtain easily that
\begin{gather} \label{e D1}
\D = \{h=h_+ + h_- : \quad h_\pm \in \D_{\max}^\pm , \quad
\Gamma_j^+ h_+ =\Gamma_j^- h_-, \ j=0,1 \ \}.
\end{gather}

Let us show that
$\{ \C^{m/2}, \Gamma_0^- , \Gamma_1^- \}$ is a boundary triple for $T_{\max}^-$.
The property (ii) of Definition \ref{dI.3} follows from the same property for the boundary triple
$\{ \C^{m/2}, \Gamma_0^+ , \Gamma_1^+ \}$ and from Proposition
\ref{p PpmIsom}~(ii).

Now we have to prove property (i) of Definition \ref{dI.3}. Since $B = B^*$, for all $f,g \in \D = \dom (B)$ we have
\begin{gather}
0= (Bf,g)_H-(f,Bg)_H
= \notag \\
\left( \ B_{\max} (P_+ f + P_- f) \ , \ P_+ g + P_- g \ \right)_H -
\left( \ P_+ f + P_- f \ , \ B_{\max} (P_+ g + P_- g) \ \right)_H
= \notag \\
\left( B_{\max}^+ P_+ f , P_+ g \right)\!_H \! +
\left( B_{\max}^- P_- f, P_- g \right)\!_H \!-
\left( P_+ f , B_{\max}^+ P_+ g \right)\!_H \!-
\left( P_- f, B_{\max}^- P_- g \right)\!_H  \label{e Bsa}
\end{gather}
Since $P_\pm f, P_\pm g \in \D_{\max}^\pm$ and
$\{ \C^{m/2}, \Gamma_0^+ , \Gamma_1^+\}$ is a boundary triple for $T_{\max}^+ = B_{\max}^+ $, Definition \ref{dI.3}
yields
\begin{multline} \label{e BP+=GP+}
\left( B_{\max}^+ P_+ f , P_+ g \right)_H -
\left( P_+ f , B_{\max}^+ P_+ g \right)_H = \\ =
(\Gamma_1^+ P_+ f, \Gamma_0^+ g)_{\C^{m/2}} -
(\Gamma_0^+ P_+ f,\Gamma_1^+ P_+ g)_{\C^{m/2}} .
\end{multline}
From \eqref{e D1} and $f,g \in \D$, we get
\begin{multline} \label{e GP=GP}
(\Gamma_1^+ P_+ f, \Gamma_0^+ P_+ g)_{\C^{m/2}} -
(\Gamma_0^+ P_+ f,\Gamma_1^+ P_+ g)_{\C^{m/2}} = \\ =
(\Gamma_1^- P_- f, \Gamma_0^- P_- g)_{\C^{m/2}} -
(\Gamma_0^- P_- f,\Gamma_1^- P_- g)_{\C^{m/2}}.
\end{multline}
It follows from \eqref{e Bsa}, \eqref{e BP+=GP+}, and \eqref{e GP=GP} that
\begin{multline*}
0=
(\Gamma_1^- P_- f, \Gamma_0^- P_- g)_{\C^{m/2}} -
(\Gamma_0^- P_- f,\Gamma_1^- P_- g)_{\C^{m/2}} + \\ +
\left( B_{\max}^- P_- f, P_- g \right)_H
 -
\left( P_- f, B_{\max}^- P_- g \right)_H . 
\end{multline*}
or, equivalently,
\begin{multline}
\left( T_{\max}^- P_- f, P_- g \right)_H
 - \left( P_- f, T_{\max}^- P_- g \right)_H = \\ =
(\Gamma_1^- P_- f, \Gamma_0^- P_- g)_{\C^{m/2}} -
(\Gamma_0^- P_- f,\Gamma_1^- P_- g)_{\C^{m/2}}
 \label{e G-+B-}
\end{multline}
for all $f,g \in \D$.
It follows easily from Proposition \ref{p PpmIsom} (ii) that the mapping $P_- :\D \to \D_{\max}^-$ is surjective. Therefore \eqref{e G-+B-} implies that
property (ii) of Definition \ref{dI.3} is fulfilled for
$\{ \C^{m/2}, \Gamma_0^- , \Gamma_1^- \}$ and so this triple is a boundary triple for $T_{\max}^-$.
Finally, note that \eqref{e D1} coincides with \eqref{e D}.
\end{proof}

\begin{proof}[Proof of Theorem \ref{p CModel}.]
By Theorem \ref{l G+=G-} (i), $n_\pm (T_{\min}^\pm) =1$ and there exist boundary triplets
$\Pi^\pm=\{\C, \Gamma_0^\pm, \Gamma_1^\pm \}$ for $T_{\max}^\pm$ such that
\eqref{e D} holds. Let $M_\pm $ be the Weyl functions of
$T_{\min}^\pm$ corresponding to the boundary triplets
$\Pi^\pm$.
Since $T_{\min}^\pm $ are densely defined operators, $M_\pm$ have the form  \eqref{e Msc} with certain constants $C_\pm \in \R_\pm$ and positive measures $d\Sigma_\pm (t)$ satisfying \eqref{e S assump}.
This fact follows from Corollary 2 in \cite[Section 1.2]{DM91} as well as from the remark after \cite[Theorem 1.1]{DM95} and \cite[Remark 5.1]{DM95} (note that, in the case of deficiency indices (1,1), condition (3) of Corollary 2 in \cite[Section 1.2]{DM91} is equivalent to the second condition in \eqref{e S assump}).
By Corollary 1 in \cite[Section 1.2]{DM91} (see also \cite[Corollary 7.1]{DM91}), the simplicity of both the operators $T_{\min}^\pm $ and $\wh T_{\Sigma_\pm}$
implies that
\begin{equation} \label{e T=UTU}
U_\pm T_{\min}^\pm U^{-1}_\pm=  T_{\Sigma_\pm} ,
\end{equation}
where $T_{\Sigma_\pm} $ are the operators defined in Subsection \ref{ss FM Sym}, and $U_\pm$ are certain unitary operators from $H_\pm$ onto $L^2 (\R, d\Sigma_\pm (t))$.
Moreover, the unitary operators $U_\pm$ can be chosen such that \begin{equation} \label{e Ga=modelGa}
\Gamma_0^\pm = \Gamma_0^{\Sigma_\pm} U_\pm , \qquad
\Gamma_1^\pm = \Gamma_1^{\Sigma_\pm, C_\pm} U_\pm .
\end{equation}
The last statement follows from the description of all possible boundary triples in terms of chosen one (see e.g. \cite{Koch75} and \cite[Proposition 1.7]{DM95}). Indeed, since the deficiency indices of $T_{\min}^\pm$ are (1,1),
\cite[formulae  (1.12) and (1.13)]{DM95} imply that
$\Gamma_0^\pm = e^{i \alpha_\pm} \, \Gamma_0^{\Sigma_\pm} U_\pm $ and
$\Gamma_1^\pm = e^{i \alpha_\pm } \, \Gamma_1^{\Sigma_\pm, C_\pm} U_\pm $
with $\alpha \in [0,2\pi)$. Now changing $U_\pm$ to $e^{i \alpha_\pm} \, U_\pm $
we save \eqref{e T=UTU} and get \eqref{e Ga=modelGa}.

Formulae \eqref{e T=} and \eqref{e D} complete the proof.
\end{proof}

\begin{rem} \label{r sa}
\textbf{(1)} Self-adjoint couplings of symmetric operators were studied in \cite{DHMS00,FHdSW05}
(see also references therein). Theorem \ref{l G+=G-} (ii) can be considered as a modification of \cite[Theorem 6.4]{DHMS00}
for $J$-self-adjoint operators. 

\textbf{(2)} Note that in Proposition \ref{p Am pm} we \emph{do not} assume
that the domain $\D_{\min}$ of $T_{\min}$ is dense in $H$. However, for convenience' sake, the operator $T_{\min}$ is assumed to be densely defined in the other statements of this subsection. The assumption $\overline{\dom (T_{\min})} = H$ can be removed
from Proposition \ref{p PpmIsom} and Theorem \ref{l G+=G-} with the use of the linear relation notion
in the way similar to  \cite[Section 6]{DHMS00}.

\textbf{(3)} Theorems \ref{l G+=G-} (ii) and \ref{p CModel} show that the operator $T$ admits an infinite family of functional models,
which corresponds to the infinite family of boundary triples. 
All the functional models can be derived from a chosen one due to \cite[Proposition 1.7]{DM95}.
\end{rem}

\subsection{The Sturm-Liouville case}
\label{ss SLModel}

Consider the differential expressions
\begin{equation}\label{II_1_01}
\lexp [y] = \frac 1{|r|} \left( -(p y')' +qy \right) \qquad \text{and}
\qquad \aexp [y] =  \frac 1{r} \left( -(py')' +qy \right),
\end{equation}
assuming that $1/p, q, r \in L_{loc}^1(a,b)$ are real-valued coefficients, that $\frac 1{p(x)}>0$ and $xr(x)>0$ for
almost all  $x\in(a,b)$, and that $-\infty \leq a <0 <b \leq + \infty$. So the weight function $r$ has the only turning point at $0$ and
the differential expressions $\aexp$ and $\lexp$ are regular at all points of the interval $(a,b)$
(but may be singular at the endpoints $a$ and $b$).
The differential expressions are understood in the sense of M.G.~Krein's  quasi-derivatives (see e.g. \cite{CL89}).

If the endpoint $a$ (the endpoint $b$) is regular or is in the limit circle case for $\lexp [ \cdot]$, we equip the expressions $\lexp [ \cdot]$ and $\aexp [\cdot]$ with a separated self-adjoint boundary condition (see e.g. \cite{Weid87} or \cite{Lan82}) at $a$ (resp., $b$), and
get in this way the self-adjoint operator $L$ and the $J$-self-adjoint operator $A$ in the Hilbert space $L^2 (\R, |r(x)|dx)$. Indeed, $A=JL$ with $J$ defined by
\begin{gather} \label{e J}
(Jf)(x) = (\sgn x) f (x) , 
\end{gather}
Obviously, $J^*=J^{-1}=J\ $ in $L^2 \left( \, (a,b), \, |r(x)|dx \right)$. So $J$ is a signature operator and $A$ is a J-self-adjoint operator.

In the case when $\lexp [\cdot]$ is in the limit point case at
$a$ and/or $b$, we do not need boundary conditions at $a$ and/or $b$.

It is not difficult to see that the operator $A_{\min} := A \cap A^*$ is a closed densely defined symmetric operator with the deficiency indices (2,2) and that $A_{\min}$ admits an orthogonal decomposition
$A_{\min} = A_{\min}^+ \oplus A_{\min}^- $, where $A_{\min}^+ $ ($A_{\min}^-$) is a part of $A_{\min} $ in $L^2 ((0,b), |r(x)|dx)$ (resp., $L^2 ((a,0), |r(x)|dx)$), see e.g. \cite[Section 2.1]{KarMMM07},  and (\ref{e Amin=}) below for a particular case (note that $A_{\min}$ \emph{is not} a minimal operator associated with $\aexp [\cdot]$ in the usual sense).
The operators $A_{\min}^\pm $ are simple. This fact considered known by specialists, it was proved in \cite{G72}, formally, under some additional conditions on the coefficients. A modification of the same proof is briefly indicated in Remark \ref{r simplicity} below.
So $A_{\min} $ is a simple symmetric operator.

Applying Theorem \ref{p CModel},
one obtains a functional model for $A$. However, we will show that
a model for $A$ can be obtained directly from the classical spectral theory of Sturm-Liouville operators and that $\Sigma_\pm$ are spectral measures associated with Titchmarsh-Weyl m-coefficients of $A$.

To avoid superfluous notation and consideration of several different cases, we argue for the case when
\begin{gather} \label{e Aas}
(a,b) = \R, \quad p \equiv 1, \quad r(x) \equiv \sgn x, \quad \\
\text{and the differential expression} \quad \lexp [\cdot] \quad \text{is limit-point at} \ +\infty \ \text{and} \ -\infty. \label{e Aas2}
\end{gather}
That is we assume that the operator
\begin{gather*} 
L=-\frac {d^2}{dx^2 } +q(x) \qquad   \left( \ A= (\sgn x) \left( -\frac {d^2}{dx^2 } +q(x) \right) \ \right)
\end{gather*}
is defined on the maximal domain and is self-adjoint (resp., J-self-adjoint). Under these assumptions,  
\begin{gather*}
\dom (L) = \dom (A) = \{ y \in L^2 (\R) \ : \ y,y' \in AC_{\loc} (\R),\ y''+qy \in L^2 (\R) \}.
\end{gather*}
The operator $A_{\min} = A \cap A^*$
has the form 
\begin{gather} 
 A_{\min}=A\upharpoonright \dom(A_{\min}), \notag \\
 \dom(A_{\min}) = \{ y \in \dom(A) \ : \ y(0)=y'(0) = 0 \}.\label{e Amin=}
\end{gather} 
By $A_{\min}^\pm$ we define the restrictions of $A_{\min}$ on $\dom(A_{\min}) \cap L^2 (\R_\pm)$.

Let us define the Titchmarsh-Weyl m-coefficients $M_{\Sl+} (\lambda)$ and $M_{\Sl-} (\lambda)$
for the Neumann problem associated with the differential
expression $ \aexp [\cdot]$ on $\R_+$ and
$\R_-$, respectively. Facts mentioned below can be found, e.g., in \cite{LevSar,Tit58}, where they are given
for spectral problems on $\R_+$, but the modification for $\R_-$ is straightforward.
Let $s (x,\lambda)$, $c (x,\lambda)$
be the solutions of the equation
\[
 \quad -y^{\prime \prime} (x)+q(x)y(x)=\lambda y(x)
\quad \]
subject to boundary conditions
\[
s ( 0, \lambda) = \frac{d}{dx} c ( 0, \lambda)=0, \qquad
\frac{d}{dx} s ( 0, \lambda) = c ( 0, \lambda)=1 .
\]
Then $M_{\Sl\pm} (\lambda)$ are well-defined by the inclusions
\begin{equation} \label{e defM}
\psi_\pm (\cdot,\la ) = - s (\cdot, \pm \lambda) +
M_{\Sl\pm} (\lambda) \, c (\cdot, \pm \lambda) \ \in \ L^2 (\R_\pm),
\end{equation}
for all $\lambda \in \C \setminus \R$.

The functions $M_{\Sl\pm} (\lambda ) $ are (R)-functions (belong to the class (R)) ;
\emph{i.e.,
$M_{\Sl\pm} (\lambda) $
are holomorphic in $\C \setminus \R$,
$M_{\Sl\pm} (\overline \lambda) = \overline{ M_{\Sl\pm} ( \lambda)}$ and
$\im \lambda \ \im M_{\Sl\pm} (\lambda) \geq 0$, $\lambda \in \C \setminus \R$} (see e.g. \cite{KK1}).

Moreover, 
$M_{\Sl\pm} (\lambda)$ admit the following representation
\begin{gather} \label{e MpmInt}
M_{\Sl\pm} (\lambda) = \int_{\R} \frac{d \Sigma_{\Sl\pm} (t)}{t-\lambda}, \
\end{gather}
where $ \Sigma_{\Sl\pm} $ are nondecreasing scalar
function such that conditions \eqref{e S assump} are fulfilled
and
\[
\displaystyle \int_\R (1+|t|)^{-1} d \Sigma_{\Sl\pm} (t) < \infty;
\]
the functions $M_{\Sl\pm} (\lambda)$ have the asymptotic formula
\begin{gather} \label{e asM}
M_{\Sl\pm} (\lambda ) = \pm \frac{i}{\sqrt{\pm\lambda}} + O\left( \frac{1}{\lambda}\right) , \quad
(\lambda \rightarrow \infty , \
0<\delta<\arg \lambda < \pi - \delta ) \ .
\end{gather}
Here and below $\sqrt{z}$ is the branch
of the multifunction on the complex plane
$\C$ with the cut along $\R_+$, singled out by the condition $\sqrt{-1}=i$.
We assume that $\sqrt{\la} \geq 0$ for $\la \in [0, +\infty)$.

Let $A_0^{\pm}$ be the self-adjoint operators
associated with the Neumann problem $y' (\pm 0) = 0$ for the differential
expression $\aexp [\cdot ]$ on $\R_\pm$.
The measures $ d\Sigma_{\Sl\pm} (t) $ are called
\emph{the spectral measures} of the operators $A_0^{\pm}$
since
\begin{equation*}
Q_{ \Sigma_{\scriptstyle \Sl\pm}} = \F_\pm A_0^{\pm} \F_\pm^{-1}
\end{equation*}
where $Q_{ \Sigma_{\scriptstyle \Sl\pm}}$ are
the operators of multiplication by t in the space
$L^2 (\R,d\Sigma_{\Sl\pm} (t))$ 
and $\F_\pm$ are the (generalized) Fourier transformations defined by
\begin{equation} \label{e F}
(\F_\pm f) (t) := \LimInMed \limits_{x_1 \rightarrow \pm \infty}
\pm \int_0^{x_1} f(x) c (x, \pm t ) dx .
\end{equation}
Here $\LimInMed $ denotes the strong
limit in $L^2 (\R,d\Sigma_{\Sl\pm})$.
Recall that $\F_\pm$ are unitary operators from $L^2 (\R_\pm)$
onto $ L^2 (\R , d\Sigma_{\Sl\pm} )$.

Note that $\supp d\Sigma_{\Sl\pm} = \sigma (Q_{\Sigma_{\scriptstyle \Sl\pm}}) = \sigma (A_0^\pm) $, that
\eqref{e MpmInt} gives a holomorphic continuation of
$\ M_{\Sl\pm} (\lambda)$ to $\ \C \setminus \supp d\Sigma_{\Sl\pm}$,
and that, in this domain, $\  M_{\Sl\pm} (\lambda) = M_{ \Sigma_{\scriptstyle \Sl\pm},
C_{\scriptstyle \Sl\pm}} (\lambda)$,
where 
\begin{gather} \label{e Cpm}
C_{\Sl\pm} := 
\int_\R \frac {t}{1+t^2} d\Sigma_{\Sl\pm} 
\end{gather}
and $M_{ \Sigma_{\scriptstyle \Sl\pm},
C_{\scriptstyle \Sl\pm}} (\lambda)$ are defined by 
(\ref{e Msc}).

\begin{thm}          \label{t SLModel}
Assume that conditions \eqref{e Aas} and \eqref{e Aas2} are fulfilled
 and the $J$-self-adjoint operator $A = (\sgn x) (-d^2/dx^2 + q(x))$ is defined as above.
Then $A $ is unitarily equivalent to the operator
$\widehat A = \widehat A \{ \Sigma_{\Sl+} , C_{\Sl+} , \Sigma_{\Sl-} , C_{\Sl-} \}$.
More precisely,
\begin{gather} \label{e FAF=whA}
(\F_+ \oplus \F_- ) A (\F_+^{-1} \oplus \F_-^{-1} ) =
\widehat A  \ .
\end{gather}
\end{thm}

\begin{proof}

The proof is based on two following  representations
of the resolvent $\Rs_{A_0^\pm}$ (see \cite{LevSar,Tit62}):
\begin{gather}
(\Rs_{A_0^\pm} (\la ) f_\pm ) (x)= \mp \psi_\pm (x,\la ) \int_0^{\pm x} \! \! \! c(s,\pm \la) f
(s) ds  \mp c(x,\pm \la) \int_{\pm x}^{\pm \infty} \! \! \! \psi_\pm (s, \la) f (s) ds, \label{e ResRepPsi}
\\
(\Rs_{A_0^\pm} (\la ) f_\pm ) (x)= \int_\R \frac{c(x,\pm t) \ (\F_\pm f_\pm) (t) \
d\Sigma_\pm (t)}{t-\la} , \quad x \in \R_\pm .\label{e ResRepF}
\end{gather}

It is not difficult to see (e.g. \cite[Section 2.1]{KarMMM07}) that
\begin{gather} \label{e domA}
\dom (A) := \left\{ y \in \dom \left( (A_{\min}^+)^* \right) \oplus
\left( (A_{\min}^-)^* \right):
y(+0) = y(-0), y'(+0) = y' (-0) \right\} .
\end{gather}

Put $\widehat A_{\min}^\pm := \F_\pm A_{\min}^\pm \F_\pm^{-1}$ and recall that
$ \widehat A_0^\pm := \F_\pm A_0^\pm \F_\pm^{-1}$ is
the operator of multiplication by $t$ in the space
$L^2 (\R,d\Sigma_{\Sl\pm} (t))$, i.e.,
$\widehat A_0^\pm = Q_{\Sigma_{\scriptstyle \Sl\pm}}$.

Let functions $f \in L^2 (\R)$ and
$f_\pm \in L^2 (\R_\pm)$ be such that
$f = f_+ + f_-$. Denote $g^\pm (t) := (\F_\pm f_\pm )(t)$.
From \eqref{e ResRepF}
we get
\begin{gather} \label{e Rf0}
(\Rs_{A_0^\pm} (\la ) f_\pm ) (\pm 0)=
\int_\R \frac{ g^\pm (t) d\Sigma_{\Sl\pm} (t)}{t-\la} .
\end{gather}
Since
$ \Rs_{\wh A_0^\pm} (\la ) g^\pm (t) =
g^\pm (t)(t-\la)^{-1}  $,
we see that
\begin{gather} \label{e y0}
y_\pm (0) = \int_\R (\F_\pm y_\pm) (t) d\Sigma_{\Sl\pm} (t) \quad
\text{ for all } \ y_\pm \in \dom (A_0^\pm), \\
\text{and} \qquad
\widehat A_{\min}^\pm  =
Q_{\Sigma_{\scriptstyle \pm}} \upharpoonright \!
\dom (\widehat A_{\min}^\pm ) \ , \quad \notag \\
\dom (\widehat A_{\min}^\pm )= \{ \widehat y_\pm \in
\dom (Q_{\Sigma_{\scriptstyle \Sl\pm}}  ) \!:\
\int_\R \widehat y_\pm(t) d\Sigma_{\Sl\pm} (t) = 0 \} . \notag
\end{gather}
That is, $\widehat A_{\min}^\pm =
 T_{\Sigma_{\scriptstyle \Sl\pm}} $.

It follows from \eqref{e ResRepPsi} that
$
\left( \Rs_{A_0^\pm} (\la) f_\pm \right) (\pm 0) = \pm \int_0^{\pm \infty}
\psi_\pm (x, \la) f_\pm (x) dx $ for $ \lambda \notin \R$.
From this and \eqref{e Rf0}, we get
\begin{gather} \label{e Fpsi}
(\F_\pm \psi_\pm (\cdot,\la)) \ (t) \ =\  \frac 1{t-\la}
\ \  \ \in L^2 ( \R, d\Sigma_{\Sl\pm} )\ .
\end{gather}

Let $ y_\pm (x) \in \dom ( (A_{\min}^\pm)^* )$.
Then, by the von Neumann formula,
\begin{gather} \label{e N}
y_\pm (t) = y_{0\pm} (t) + c_1 \psi_\pm (t,i)+ c_2 \psi_\pm (t,-i) \ ,
\end{gather}
where $ y_{0\pm} (t) \in \dom ( A_{\min}^\pm )$
and $c_1, c_2 \in \C$ are certain constants.
Therefore \eqref{e defM} yields
\begin{gather*} 
y_\pm (0) = c_1 M_{\Sl\pm} (i)+ c_2 M_{\Sl\pm} (-i)
= c_1 \int_\R \frac 1{t-i} \ d\Sigma_{\Sl\pm} (t) +
 c_2 \int_\R \frac 1{t+i} \ d\Sigma_{\Sl\pm} (t) .
\end{gather*}
This, \eqref{e Fpsi}, and \eqref{e y0} implies that \eqref{e y0} holds for all
$ y_\pm (x) \in \dom ( (A_{\min}^\pm)^* )$. Taking \eqref{e Cpm} and \eqref{e GammaS} into account, we get
\begin{gather} \label{e G1}
y_\pm (0) = \int_\R (\F_\pm y_\pm) (t) \ d\Sigma_\pm (t)
= 
\Gamma_1^{\Sigma_{\Sl\pm}, C_{\Sl\pm}}  \F_\pm y_\pm \ .
\end{gather}
Further,
by \eqref{e N},
$y^\prime_\pm (0) = -c_1 - c_2 $.
On the other hand, it follows from $\widehat A_{\min}^\pm =
 T_{\Sigma_{\scriptstyle \Sl\pm}} $ and \eqref{e GammaS}
that $\Gamma_0^{\Sigma_{\Sl\pm}} \F y_{0\pm} = 0 $ and
$\Gamma_0^{\Sigma_{\Sl\pm}} (t-\la)^{-1} = 1$. 
Hence,
\begin{gather} \label{e G2}
\Gamma_0^{\Sigma_{\Sl\pm}} \F y_\pm = c_1 \Gamma_0^{\Sigma_{\Sl\pm}} \frac 1{t-i}+
c_2 \Gamma_0^{\Sigma_{\Sl\pm}} \frac 1{t+i} = c_1 + c_2 =
-y_\pm ^\prime (0).
\end{gather}
Combining \eqref{e domA}, \eqref{e G1}, and \eqref{e G2},
we get \eqref{e FAF=whA}.
\end{proof}

\begin{rem} \label{r simplicity}
Since the operators $T_{\Sigma_{\scriptstyle \Sl\pm}}$ are simple (see \cite[Proposition 7.9]{MM03}),
in passing it is proved that so are the operators $ A_{\min}^\pm $ and $ A_{\min}$. This proof of simplicity works in general case of Sturm-Liouville operator with
one turning point described in the beginning of this section. Formally, it removes extra smoothness assumptions on the coefficient $p$ imposed in
\cite{G72}. But actually it is just another version of the proof of \cite[Theorem 3]{G72} since the essence of both the proofs is based on Krein’s criterion for simplicity \cite[Section 1.3]{K49}.
\end{rem}

\section{Point and essential spectra of the model operator $\widehat A$ and of indefinite Sturm-Liouville operators \label{s point s}}

\subsection{Point spectrum of the model operator \label{ss point model}}

The main result of this section and of the paper is a description of the point spectrum and algebraic multiplicities of eigenvalues
of $\widehat A \{ \Sigma_+ , C_+ , \Sigma_- , C_- \}$.

First, to classify eigenvalues of the operator $T_\Sigma^*$ defined in Subsection \ref{ss FM Sym}, we introduce the following mutually
disjoint sets:
\begin{gather}
\mathfrak{A}_0 (\Sigma) = \left\{ \lambda \in \sigma_c (Q_\Sigma) \ :
\ \int_\R |t-\lambda|^{-2} d\Sigma (t) = \infty \right\} , \label{e fA0}\\
\mathfrak{A}_r (\Sigma) = \left\{ \lambda \not \in \sigma_p (Q_\Sigma) \ :
\ \int_\R |t-\lambda|^{-2} d\Sigma (t) <\infty \right\}, \qquad \notag\\
\mathfrak{A}_p (\Sigma) = \sigma_p (Q_\Sigma). \notag
\end{gather}
Observe that $\C = \mathfrak{A}_0 (\Sigma) \cup \mathfrak{A}_r (\Sigma) \cup
\mathfrak{A}_p (\Sigma) $ and
\begin{gather}
\mathfrak{A}_0 (\Sigma) = \left\{ \lambda \in \C \ :
\ \ker ( T_\Sigma^* - \lambda I) = \{ 0 \} \ \right\} ,
\notag
\\
\mathfrak{A}_r (\Sigma) = \left\{ \lambda \in \C :
\ \ker ( T_\Sigma^* - \lambda I) = \{ c(t-\lambda)^{-1}, \; c \in \C \} \right\}, \label{e Arker}
\\
\mathfrak{A}_p (\Sigma) = \left\{ \lambda \in \C \ :
\ \ker ( T_\Sigma^* - \lambda I) = \{ c\chi_{ \{\lambda\} } (t), \; c \in \C \} \label{e Apker}
\right\} .
\end{gather}

In this section we  denote for brevity $\Gamma_0^\pm :=\Gamma_0^{\Sigma_\pm}$,
$\Gamma_1^\pm :=\Gamma_1^{\Sigma_\pm,C_\pm}$, where  $\Gamma_0^{\Sigma_\pm}$, $\Gamma_1^{\Sigma_\pm,C_\pm}$ are linear mappings from $\dom(T_{\Sigma_\pm}^*)$ to $\C$ defined by \eqref{e GammaS}.

In this paper, for fixed $\lambda \in \R$, the notation $\frac {\chi_{ \R \setminus \{\la \}} (t)}{ (t-\la)^j } $ means the function that is equal to $0$ at $t=\lambda$ and $\frac 1{(t-\la)^j}$ for $t \neq \lambda$.
If $\lambda \not \in \R$, then $\frac {\chi_{ \R \setminus \{\la \}} (t)}{ (t-\la)^j } $ means just $\frac 1{(t-\la)^j}$.
In what follows the functions $\frac {\chi_{ \R \setminus \{\la \}} (t)}{ (t-\la)^j } $ and jump discontinuities of $\Sigma$ play an essential role.
Note that the set of jump discontinuities of $\Sigma$ coincides with $\mathfrak{A}_p (\Sigma) = \sigma_p (Q_\Sigma)$. If $\la \in \R \setminus \mathfrak{A}_p (\Sigma)$,  then $\frac {\chi_{ \R \setminus \{\la \}} (t)}{ (t-\la)^j } $ and $\frac 1{(t-\la)^j}$ belong to the same class of $L^2 (\R,d\Sigma)$ and any
of these two notations can be used. We also use notation $ d \Sigma (\{\lambda \}) := \Sigma (\lambda+0) - \Sigma (\lambda -0)$.

For the sake of simplicity,
we start from the case when
\begin{gather} \label{e CpmCond}
\int_\R (1+|t|)^{-1} d\Sigma_\pm < \infty \quad \text{and} \quad
C_\pm = \int_\R t(1+t^2)^{-1} d\Sigma_\pm,
\end{gather}
(which arises, in particular, in Section \ref{ss SLModel}) and then consider the general case.

\begin{thm} \label{t s p}
Let $\Sigma_\pm$ be nondecreasing scalar functions satisfying
\eqref{e S assump} and let $C_\pm$ be real constants.
Assume also that conditions \eqref{e CpmCond} are fulfilled. Then
the following statements describe the point spectrum of the operator
$\widehat A = \widehat A \{ \Sigma_+ , C_+ , \Sigma_- , C_- \}$.
\begin{description}
\item[1)]
If $\lambda \in \mathfrak{A}_0 (\Sigma_+) \cup \mathfrak{A}_0 (\Sigma_-)$,
then $ \lambda \not \in \sigma_p (\widehat A)$.
\item[2)]
If $\lambda \in \mathfrak{A}_p (\Sigma_+) \cap \mathfrak{A}_p (\Sigma_-)$,
then
\begin{description}
\item[(i)] $\lambda$ is an eigenvalue of $\widehat A$; the geometric multiplicity of $\lambda$ equals 1;
\item[(ii)] the eigenvalue $\lambda$ is simple (i.e., the algebraic and geometric multiplicities are equal to 1) if an only if at least one of the following conditions is not fulfilled:
\begin{gather} \label{e S-S=S-S}
d\Sigma_- (\{ \lambda \} ) = d\Sigma_+ ( \{ \lambda \} ),
\\      
\label{e Int+<Inf}
\int_{\R \setminus \{\lambda \}} \frac 1{|t-\lambda |^2} \ d \Sigma_+ (t) < \infty ,
\\      
\label{e Int-<Inf}
\int_{\R \setminus \{\lambda \}} \frac 1{|t-\lambda |^2} \ d \Sigma_- (t) < \infty ;
\end{gather}
\item[(iii)] if conditions \eqref{e S-S=S-S}, \eqref{e Int+<Inf} and
\eqref{e Int-<Inf} hold true, then the algebraic multiplicity of $\lambda$ equals
the greatest number $k$ ($k \in \{2,3,4,\dots\} \cup \{ +\infty \}$) such that the conditions
\begin{gather} \label{e frj<inf}
\int_{\R \setminus \{\lambda \}} \frac 1{|t-\lambda|^{2j}} \ d \Sigma_- (t) < \infty ,
\qquad
\int_{\R \setminus \{\lambda \}} \frac 1{|t-\lambda|^{2j}} \ d \Sigma_+ (t) < \infty ,
\\            
 \label{e frj=frj}
\int_{\R \setminus \{\lambda \}} \frac 1{(t-\lambda )^{j-1}} \ d\Sigma_- (t) =
\int_{\R \setminus \{\lambda \}} \frac 1{(t-\lambda )^{j-1}} \ d\Sigma_+ (t) ,
\end{gather}
are fulfilled for all natural $ j $ such that $ 2 \leq j \leq k-1 $
(in particular, $k=2$ if at least one of conditions \eqref{e frj<inf}, \eqref{e frj=frj} is not fulfilled for $j=2$).
\end{description}
\item[3)]
Assume that
$\lambda \in \mathfrak{A}_r (\Sigma_+) \cap \mathfrak{A}_r (\Sigma_-)$.
Then $\lambda \in \sigma_p (\widehat A)$ if and only if
\begin{gather} \label{e fr1=fr1}
\int_\R \frac 1{t-\lambda} d\Sigma_+ (t) =
\int_\R \frac 1{t-\lambda} d\Sigma_- (t) \ .
\end{gather}
If \eqref{e fr1=fr1} holds true, then the geometric multiplicity of $\lambda$
is 1, and the algebraic multiplicity is the greatest number $k$
($k \in \{1,2,3,\dots\} \cup \{ +\infty \}$) such that
the conditions
\begin{gather}
\int_{\R } \frac 1{|t-\lambda |^{2j}}
\ d \Sigma_- (t) < \infty ,
\qquad
\int_{\R } \frac 1{|t-\lambda |^{2j}} \ d \Sigma_+ (t) < \infty ,
\label{e frj<inf r}
\\
\int_{\R } \frac 1{(t-\lambda )^{j}} \ d\Sigma_- (t) =
\int_{\R } \frac 1{(t-\lambda )^{j}} \ d\Sigma_+ (t)
\label{e frj=frj r}
\end{gather}
are fulfilled for all $ j \in \N$ such that $ 1 \leq j \leq k $.
\item[4.]
If $\lambda \in \mathfrak{A}_p (\Sigma_+) \cap \mathfrak{A}_r (\Sigma_-)$
or $\lambda \in \mathfrak{A}_p (\Sigma_-) \cap \mathfrak{A}_r (\Sigma_+)$,
then $\lambda \not \in \sigma_p (\widehat A)$.
\end{description}
\end{thm}

\begin{proof}
A vector
$y = \left( \begin{array}{c} y_- \\ y_+ \end{array} \right) \in L^2(\R,d\Sigma_+)\oplus L^2(\R,d\Sigma_-)$
is a solution of the equation
$\wh A y = \la y$ if and only if
\begin{gather*} 
 y \ \in \ \ker (T_{\Sigma_+} ^* -\la I) \oplus
\ker (T_{\Sigma_-}^* -\la I) \quad \text{and} \quad
y \in \dom (\wh A).
\end{gather*}
Recall that $h = \left( \begin{array}{c} h_- \\ h_+ \end{array} \right) \in \dom (T_{\Sigma_+}^*) \oplus \dom (T_{\Sigma_-}^*)$ belongs to $\dom (\wh A)$ if and only if
\begin{gather}
 \label{e hBC}
\Gamma_0^- h_-
= \Gamma_0^+ h_+ \  , \quad \Gamma_1^- h_-
= \Gamma_1^+ h_+ \ .
\end{gather}

It follows from (\ref{e GammaS}) that
\begin{gather} \label{e whG1gen}
\Gamma_1^\pm h_\pm = C_\pm \Gamma_0^\pm h_\pm+
\int_\R \left( h_\pm (t) - \frac{t\Gamma_0^\pm h_\pm}{t^2+1} \right) d\Sigma_\pm (t)  , \quad h_\pm \in \dom (T_{\Sigma_\pm}^*) .
\end{gather}
(\ref{e CpmCond}) and (\ref{e GammaS}) yield $ h_\pm (t) \in L^1 (\R, d\Sigma_\pm)$ for arbitrary $ h_\pm (t) \in \dom ( T_{\Sigma_\pm} ^*) $, and using (\ref{e whG1gen}), we obtain
\begin{gather} \label{e whG1}
\Gamma_1^\pm y_\pm = \int_\R y_\pm (t) d\Sigma_\pm (t) \ .
\end{gather}

If $\la \in \mathfrak{A}_r $, then $\frac1{t-\la} \in L^2 (\R, d\Sigma_\pm )$ and \eqref{e GammaS} (or even simpler \eqref{e hatA*})
yields that $\frac1{t-\la} \in \dom ( T_{\Sigma_\pm} ^*)$ and
\begin{gather} \label{e whG0}
\Gamma_0^\pm \frac 1{t -\la} = 1 \ .
\end{gather}
The function $ \chi_{\{\la\}} (t)$, $\lambda \in \R$, is a nonzero vector
in $L^2 (\R, d\Sigma^\pm )$ exactly when $\la \in \mathfrak{A}_p$; in this case,
\begin{gather} \label{e Gev12} 
\Gamma_0^\pm \chi_{\{\la\}} = 0 , \qquad
\Gamma_1^\pm \chi_{\{\la\}} =
\int_{ \{\la \} } d\Sigma_\pm (t) = d\Sigma_\pm (\{\la\}) \ .
\end{gather}

\textbf{1)} \ Suppose $\wh A y=\la y$ and consider the case
$\ker (  T_{\Sigma_-}^* - \la) = \{ 0 \} $
(the case $\ker ( T_{\Sigma_+}^* - \la) = \{ 0 \} $
is analogous).
Then $ y_- = 0$ and, by \eqref{e hBC}, we get
$ \Gamma_0^+ y_+ =0$,
$ \Gamma_1^+ y_+ =0$.
Hence $y_+ \in \dom ( Q_{\Sigma_+})$ (see (\ref{e GammaS})),  and
$ Q_{\Sigma_+} y_+ = \la y_+$.
This implies  $y_+ (t) = c_1 \chi_{\{\la\}} (t)$,
$c_1 \in \C$. On the other hand, $0=\Gamma_1^+ y_+ (t) =
\int_\R y_+ (t) d\Sigma_+ (t) $.
Thus $c_1 = 0$ and $y_+ = 0$ a.e. with respect to the measure
$d\Sigma_+$.

\textbf{2)} \ Let
$\lambda \in \mathfrak{A}_p (\Sigma_+) \cap \mathfrak{A}_p (\Sigma_-)$.
By \eqref{e Apker}, we have
\[
y (t)=
\left( \begin{array}{c}
c_1^- \chi_{\{\la \}} (t) \\ c_1^+ \chi_{\{\la \}} (t)
\end{array} \right), \quad c_1^\pm \in \C.
\]
Since
$\lambda \in \mathfrak{A}_p (\Sigma_\pm) $, we see that $\la \in \R$ and  $ d\Sigma_\pm ( \{\la \}) \neq 0$.
Taking into account \eqref{e Gev12}, we see that
system \eqref{e hBC} is equivalent to
$ c_1^- d\Sigma_- (\{\la \}) =
 c_1^+ d\Sigma_+ (\{\la \})$.
Therefore the geometric multiplicity of $\lambda$ equals 1 and
\begin{equation} \label{e y0 c2}
\wh y_0 =  \left( \begin{array}{c}
\frac 1{ d\Sigma_- (\{\la \})} \chi_{\{\la \}} (t) \\
\frac 1{ d\Sigma_+ (\{\la \})} \chi_{\{\la \}} (t)
\end{array} \right) \quad \text{ is one of corresponding eigenvectors of}
\ \wh A .
\end{equation}

Let
$ y_1 = \left( \begin{array}{c}
 y_1^- \\ y_1^+
\end{array} \right)$
and
$\wh A y_1 - \la  y_1 = y_0 $.
By \eqref{e hatA*}, we have
\[
\left( \begin{array}{c}
t y_1^- (t) \\ t y_1^+ (t)
\end{array} \right)
-
\left( \begin{array}{c}
\Gamma_0^- y_1^- \\ \Gamma_0^+ y_1^+
\end{array} \right) - \la
\left( \begin{array}{c}
 y_1^- (t) \\  y_1^+ (t)
\end{array} \right)
=
 y_0 \ .
\]
Thus,
\[
(t-\la ) y_1^\pm (t) = \frac 1{d\Sigma_\pm (\{\la\})}
 \chi_{ \{\la \}} (t) + \Gamma_0^\pm y_1^\pm \ .
\]
Choosing $t=\la$, we obtain
\begin{gather} \label{e G0=c=S}
\Gamma_0^\pm y_1^\pm =
- \frac 1{d\Sigma_\pm (\{\la\})} \
\neq 0.
\end{gather}
Therefore,
\begin{gather} \label{e y1p}
y_1^\pm = -\frac 1{ d\Sigma_\pm (\{\la \})} \:
\frac {\chi_{ \R \setminus \{\la \}}(t)}{ t-\la }   +
c_2^\pm \chi_{ \{\la \}} (t)\ ,
\end{gather}
where $ c_2^\pm \in \C$.
The conditions $y_1^+ \in L^2 (\R, d\Sigma_+ )$ and
$y_1^- \in L^2 (\R, d\Sigma_- )$ are
equivalent to (\ref{e Int+<Inf}) and (\ref{e Int-<Inf}),
respectively.

Assume that \eqref{e Int+<Inf} and \eqref{e Int-<Inf}
are fulfilled.
By \eqref{e whG1}, we have
\begin{gather*} 
\Gamma_1^\pm y_1^\pm = - \frac 1{ d\Sigma_\pm (\{\la\}) }
\int_{\R \setminus \{\la \}} \frac 1{t-\la} \ d \Sigma_\pm (t)
+ 
c_2^\pm  d\Sigma_\pm (\{\la\}).
\end{gather*}
The latter and \eqref{e G0=c=S} implies that $\wh y_1 \in \dom (\wh A)$
if and only if the conditions
\eqref{e S-S=S-S} and
\begin{gather}
- \frac 1{ d\Sigma_- (\{\la\}) }
\int_{\R \setminus \{\la \}} \frac 1{t-\la} \ d \Sigma_- (t) +
c_2^-  d\Sigma_- (\{\la \}) = \notag \\ =
- \frac 1{ d\Sigma_+ (\{\la \}) }
\int_{\R \setminus \{\la \}} \frac 1{t-\la} \ d \Sigma_+ (t) +
c_2^+  d \Sigma_+ (\{\la \}) \label{e hBC1 p}
\end{gather}
are fulfilled.
Thus, the quotient space $\ker (\wh A - \la)^2 / \ker (\wh A - \la)
\neq \{ 0 \}$
if and only if the conditions \eqref{e S-S=S-S}, \eqref{e Int+<Inf},
and \eqref{e Int-<Inf} are satisfied.
In this case, generalized eigenvectors of first order
$ y_1$ have the form (\ref{e y1p}) with constants
$c_2^\pm $ such that (\ref{e hBC1 p}) holds.

Assume that all condition mentioned above are satisfied.
Then $ \dim \ker (\wh A - \la)^2 / \ker (\wh A - \la)= 1$
and one of generalized eigenvectors of first order is given by the constants
\[
c_2^\pm = -\alpha_1^2
\int_{\R \setminus \{ \la \}} \frac 1{t-\la} \ d \Sigma_\mp (t) ,
\]
where
$ \quad
\alpha_1 := \frac 1{d\Sigma_- (\{\la \})} =
\frac 1{d\Sigma_+ (\{ \la \}) } \ .
$

If $ y_2  = \left( \begin{array}{c}
y_2^- \\ y_2^+
\end{array} \right)$ and
$\wh A y_2 -\la y_2 = y_1$,
then
\begin{equation} \label{t-la y2}
(t-\la ) y_2^\pm (t) \, = \, y_1^\pm (t) + \Gamma_0^\pm y_2^\pm
\, = \,
- \alpha_1 \frac {\chi_{ \R \setminus \{\la \}} (t)}{ t-\la }   +
c_2^\pm \chi_{ \{\la \}} (t) + \Gamma_0^\pm \, y_2^\pm \ .
\end{equation}
For $t=\la$ we have
\begin{gather} \label{e t=la2 p}
\Gamma_0^\pm y_2^\pm = - c_2^\pm .
\end{gather}
Consequently,
\begin{gather} \label{e y2p}
y_2^\pm = - \alpha_1 \frac {\chi_{ \R \setminus \{\la \}}(t)}{ (t-\la)^2 }
- c_2^\pm \frac {\chi_{ \R \setminus \{\la \}}(t)}{ t-\la } +
c_3^\pm \chi_{ \{\la \}} (t) , \quad c_3^\pm \in \C.
\end{gather}
By (\ref{e t=la2 p}), conditions (\ref{e hBC})
for $ y_2$ has the form
\begin{gather} \label{e hBC22 p}
c_2^- = c_2^+ \ ,
\end{gather}
\begin{gather*} 
-\alpha_1 \int_{\R \setminus \{\la \}} \frac 1{(t-\la)^2} \ d \Sigma_- (t) -
c_2^- \int_{\R \setminus \{\la \}} \frac 1{t-\la} \ d \Sigma_- (t) +
c_3^-  \alpha_1^{-1}
= \\ =
-\alpha_1 \int_{\R \setminus \{\la \}} \frac 1{(t-\la)^2} \ d \Sigma_+ (t) -
c_2^+ \int_{\R \setminus \{\la \}} \frac 1{t-\la} \ d \Sigma_+ (t) +
c_3^+  \alpha_1^{-1} .
\end{gather*}
Thus $y_2$ exists if and only if
$\dfrac {\chi_{ \R \setminus \{\la \}}(t)}{ (t-\la)^2 } \in
L^2 (\R, d\Sigma_\pm )$ and \eqref{e hBC22 p} is fulfilled.
This is equivalent \eqref{e frj<inf} and \eqref{e frj=frj} for $j=2$.

Continuing this line of reasoning, we obtain part 2) of the theorem.

3) \ The idea of the proof for part 3) is
similar to that of part 2), but calculations are simpler.
Let
$\lambda \in \mathfrak{A}_r (\Sigma_+)
\cap \mathfrak{A}_r (\Sigma_+)$.
Then
$ \displaystyle \quad y (t)=
\left( \begin{array}{c}
c_1^- \frac 1{t-\la} \\ c_1^+ \frac 1{t-\la}
\end{array} \right)  .
$
Hence \eqref{e hBC} has the form
\[
c_1^- = c_1^+ ,
\qquad
c_1^- \int_\R \frac 1{t-\la} d \Sigma_- (t) =
c_1^+ \int_\R \frac 1{t-\la} d \Sigma_+ (t) \ .
\]
Consequently $\la$ is an eigenvalue of $\wh A$
if and only if \eqref{e fr1=fr1} holds true;
in this case the geometric multiplicity is 1
and
$ y_0 =  \left( \begin{array}{c}
\frac 1{t-\la} \\ \frac 1{t-\la}
\end{array} \right) $
is a corresponding eigenvector of $\wh A$.

Let $\wh A y_1 - \la  y_1 = y_0 $ where
 $ y_1 = \left( \begin{array}{c}
 y_1^- \\  y_1^+
\end{array} \right)$.
Then
\[
(t-\la ) y_1^\pm (t) = \frac 1{t-\la} + c_2^\pm,
\qquad c_2^\pm= \Gamma_0^\pm  y_1^\pm .
\]
Therefore
$ 
y_1^\pm =
\frac 1{ (t-\la )^2 }   + \frac { c_2^\pm }{ t-\la } . \
$
The case $ y_1^\pm \in L^2 (\R, d\Sigma_\pm )$
is characterized by \eqref{e frj<inf r} with $j=2$.
Conditions \eqref{e hBC} become
\begin{gather*}
\int_{\R } \left( \frac 1{(t-\la)^2} + \frac {c_2^-}{t-\la} \ \right)
d \Sigma_- (t) =
\int_{\R } \left( \frac 1{(t-\la)^2} + \frac {c_2^+}{t-\la} \  \right)
d \Sigma_+ (t) ,
\end{gather*}
\[
c_2^- =c_2^+ \ .
\]
Taking into account \eqref{e fr1=fr1}, we see
that the generalized eigenvector $\wh y_1$ exists if and only if
 conditions \eqref{e frj<inf r}, \eqref{e frj=frj r} are
satisfied for $j=2$.  Continuing this line of reasoning, we obtain part 3) of the theorem.

4) \ Suppose
$\lambda \in \mathfrak{A}_p (\Sigma_+)
\cap \mathfrak{A}_r (\Sigma_-)$
(the case $\lambda \in \mathfrak{A}_p
(\Sigma_-) \cap \mathfrak{A}_r (\Sigma_+)$ is similar).
Then \quad
$ \displaystyle \wh y (t)=
\left( \begin{array}{c}
c_1^- \frac 1{t-\la} \\ c_1^+ \chi_{\{\la \}} (t)
\end{array} \right)
$
and \eqref{e hBC} has the form
\[
c_1^- =0, \qquad c_1^- \int_\R \frac 1{t-\la} d \Sigma_- (t) =
c_1^+ d\Sigma_+ (\{\la \} )  \ .
\]
Thus $ c_1^- = c_1^+ =0$ and $\la \not\in \sigma_p (\wh A)$.

\end{proof}

Now we consider the general case when the functions $\Sigma_\pm $
satisfy (\ref{e S assump}) and $C_\pm$ are arbitrary real constants.

\begin{lem} \label{l eq cond}
Let $k \in \N$ and let one of the following two assumptions be fulfilled:
\begin{description}
\item[(a)] $\lambda \in \C\setminus \R$ \quad or
\item[(b)] $\lambda \in \R $, \quad  $d\Sigma_+ (\{ \lambda \}) = d\Sigma_- (\{ \lambda \})$,  \quad
$\frac { \chi_{\R \setminus \{\lambda\} }(t) } {(t-\lambda )^{k}} \in
L^2 (\R, d\Sigma_+)$, \quad  \\
and $\frac { \chi_{\R \setminus \{\lambda\} }(t) } {(t-\lambda )^{k}} \in L^2 (\R, d\Sigma_-)$.
\end{description}
Then $\frac { \chi_{\R \setminus \{\lambda\} } (t)} {(t-\lambda )^{k}} \in \dom (T_{\Sigma_+}^*)$, \quad  $\frac { \chi_{\R \setminus \{\lambda\} } (t)} {(t-\lambda )^{k}} \in \dom (T_{\Sigma_-}^*)$,
and the following two statements are equivalent:
\begin{description}
\item[(i)] $ \Gamma_1^- \frac { \chi_{\R \setminus \{\lambda\} } (t)} {(t-\lambda )^{k}}
= \Gamma_1^+ \frac { \chi_{\R \setminus \{\lambda\} } (t)} {(t-\lambda )^{k}} $;
\item[(ii)] $ \lim\limits_{\substack{\ep \to 0 \\ \ep \in \R}} \MpmMm^{(k-1)} (\lambda+i\ep) = 0$, where the function $ \MpmMm $ is defined by \\
$ \MpmMm := M_{\Sigma_+,C_+} - M_{\Sigma_-,C_-} $ and $ \MpmMm^{(j)} $ is its $j$-th derivative ($ \MpmMm^{(0)} = \MpmMm $).
\end{description}

If, additionally, $\lambda \not \in \sigma_{\ess} (Q_{\Sigma_+}) \cup \sigma_{\ess} (Q_{\Sigma_-})$, then statements (i) and (ii) are equivalent to
\begin{description}
\item[(iii)] the function
$\MpmMm $ is analytic in a certain neighborhood of $\lambda$ and \\
$ \MpmMm^{(k-1)} (\lambda) = 0 $. (If $M_{\Sigma_+,C_+} - M_{\Sigma_-,C_-}$ is defined in a punctured neighborhood of $\lambda$ and
has a removable singularity at $\lambda$, then we assume that $\MpmMm$ is analytically extended over $\lambda$.)
\end{description}
\end{lem}

\begin{proof}
We assume here and below that $j \in \N$.

First note that if $\lambda \not \in \sigma_{ess} (Q_{\Sigma_\pm})$, then $\frac { \chi_{\R \setminus \{\lambda\} } } {(t-\lambda )^{j}} \in L^2 (\R, d\Sigma_\pm)$ for any $j \in \N$, and using the definition of $\dom ( T_{\Sigma_\pm}^*)$, we see that
$\frac { \chi_{\R \setminus \{\lambda\} } } {(t-\lambda )^{j}} \in \dom ( T_{\Sigma_\pm}^*)$ for any $j$.

Generally, the last statement is not true for $\lambda \in \sigma_{ess} (Q_{\Sigma_\pm})$. But under assumptions of the lemma, we have
\begin{equation} \label{e -j<inf pm}
\int_{\R \setminus \{ \lambda \} } \frac{1}{|t-\lambda|^{2j}} d\Sigma_\pm (t) < \infty
\end{equation}
for $j = k$. Taking into account the first assumption in (\ref{e S assump}), we see that (\ref{e -j<inf pm}) is valid for
all $j \leq k$. The latter implies that
$\frac { \chi_{\R \setminus \{\lambda\} } (t)} {(t-\lambda )^{j}} \in \dom ( T_{\Sigma_\pm}^*)$ for
all $j \leq k$.
Moreover,
$\frac { \chi_{\R \setminus \{\lambda\} } (t)} {(t-\lambda )^{j}} \in \dom (Q_{\Sigma_\pm})$ if $2 \leq j \leq k$ (assuming $k \geq 2$). Therefore,
\begin{equation*}
\Gamma_0^\pm \ \frac { \chi_{\R \setminus \{\lambda\} } (t)} {(t-\lambda )^{j}} = 0 , \quad  2 \leq j \leq k \ .
\end{equation*}
The last statement does not hold in the case $j=1$. Using (\ref{e hatA*}), one has
\begin{equation} \label{e chi1}
\Gamma_0^\pm \ \frac { \chi_{\R \setminus \{\lambda\} } (t)} {t-\lambda } = 1.
\end{equation}
Eqs.~(\ref{e GammaS}) (see also (\ref{e whG1gen})) allow us to conclude that

\begin{eqnarray}
& \Gamma_1^\pm \ \frac { \chi_{\R \setminus \{\lambda\} } (t) } {t-\lambda } =
C_\pm + \int_\R \left( \frac { \chi_{\R \setminus \{\lambda\} } (t)} {t-\lambda } - \frac t{t^2+1}  \right) d\Sigma_\pm (t) , \label{e G1pm 1}\\
& \Gamma_1^\pm \ \frac { \chi_{\R \setminus \{\lambda\} } (t)} {(t-\lambda )^{j}} =
\int_\R \ \frac { \chi_{\R \setminus \{\lambda\} } (t)} {(t-\lambda )^j} \ d\Sigma_\pm (t) \ \qquad \text{if} \quad 2 \leq j \leq k \ . \label{e G1pm k}
\end{eqnarray}

If $\lambda \not \in \sigma(Q_{\Sigma_\pm})$ (in particular, if $\lambda \not \in \R$),
then (\ref{e Msc}) shows that
\[
\Gamma_1^\pm \ \frac { \chi_{\R \setminus \{\lambda\} } (t)} {(t-\lambda )^{j}} = \Gamma_1^\pm \ \frac { 1} {(t-\lambda )^{j}} = (j-1)! \ M_{\Sigma_\pm,C_\pm}^{(j-1)} (\lambda).
\]
 This prove the equivalence of (i), (ii),
and (iii) for the case when $\lambda \not \in \sigma(Q_{\Sigma_+}) \cup \sigma(Q_{\Sigma_-})$ (this simplest case explains the crux of the lemma).

Consider the case $\lambda \in \sigma(Q_{\Sigma_+}) \cup \sigma(Q_{\Sigma_-})$ and $\lambda \not \in \sigma_{\ess} (Q_{\Sigma_+}) \cup \sigma_{\ess} (Q_{\Sigma_-})$.
The assumptions of the lemma state that $d\Sigma_+ (\{ \lambda \}) = d\Sigma_- (\{ \lambda \})$.
So $\lambda $ is an isolated eigenvalue of both the operators $ Q_{\Sigma_+}$ and $ Q_{\Sigma_-}$
and is an isolated jump discontinuity of $\Sigma_+$ and $\Sigma_-$. This and
$d\Sigma_+ (\{ \lambda \}) = d\Sigma_- (\{ \lambda \})$ imply that $\MpmMm $ has a removable singularity at $\lambda$ and can be considered as an analytic function in a certain neighborhood of $\lambda$. Moreover,
\[
(k-1)! \ \MpmMm^{(k-1)} (\lambda) = \Gamma_1^+ \frac { \chi_{\R \setminus \{\lambda\} } } {(t-\lambda )^{k}} - \Gamma_1^- \frac { \chi_{\R \setminus \{\lambda\} } } {(t-\lambda )^{k}},
\]
 and (i) $\Leftrightarrow$ (ii) $\Leftrightarrow$
(iii) is shown again.

Now let assumption (b) be satisfied and let $\lambda \in \sigma_{\ess} (Q_{\Sigma_+}) \cup \sigma_{\ess} (Q_{\Sigma_-})$.
Then the function $\MpmMm $ is not analytic in $\lambda$, but the limit in statement (ii) exists and
\begin{equation} \label{e MpmMm=G-G}
(k-1)! \ \lim_{\substack{\ep \to 0 \\ \ep \in \R}} \MpmMm^{(k-1)} (\lambda+i\ep) =
\Gamma_1^+ \frac { \chi_{\R \setminus \{\lambda\} } (t)} {(t-\lambda )^{k}} - \Gamma_1^- \frac { \chi_{\R \setminus \{\lambda\} } (t)} {(t-\lambda )^{k}} \ .
\end{equation}
Indeed, taking $d\Sigma_+ (\{ \lambda \}) = d\Sigma_- (\{ \lambda \})$ into account, we get
$\MpmMm (z) = C_+ - C_- + \I_{\infty} (z) + \I_{\lambda} (z) $, where
\begin{eqnarray*} 
& \I_{\infty} (z) := \int\limits_{\R \setminus [\lambda -\delta, \lambda +\delta ] } \left( \frac 1{t-\lambda } - \frac t{t^2+1}  \right) \left( d\Sigma_+ (t) - d\Sigma_- (t) \right) \ , \\
& \I_{\lambda} (z) := \int\limits_{[\lambda -\delta , \lambda) \cup (\lambda, \lambda + \delta ] }
 \left( \frac 1{t-\lambda } - \frac t{t^2+1}  \right) \left( d\Sigma_+ (t) - d\Sigma_- (t) \right) \ ,
\end{eqnarray*}
and $\delta $ is any fixed positive number.
The function $\I_{\infty} (z) $ is analytic at $\lambda$.
Formula (\ref{e -j<inf pm}) is valid for $j \leq k$ and allows us to apply Lebesgue's dominated convergence theorem to the limit
$\lim\limits_{\substack{\ep \to 0 \\ \ep \in \R}} \ \I_{\lambda}^{(j-1)} \, (\lambda + i \ep) $.
As a result, we see that  (\ref{e G1pm 1}) implies (\ref{e MpmMm=G-G}) for $k=1$ and (\ref{e G1pm k}) implies (\ref{e MpmMm=G-G})
for $k \geq 2$.
\end{proof}

\begin{thm} \label{t s p GC}
Let $\widehat A = \widehat A \{ \Sigma_+ , C_+ , \Sigma_- , C_- \}$, where the functions $d \Sigma_\pm $
satisfy (\ref{e S assump}) and $C_\pm$ are certain real constants.
Then the following statements hold:
\begin{description}
\item[1)]
If $\lambda \in \mathfrak{A}_0 (\Sigma_+) \cup \mathfrak{A}_0 (\Sigma_-)$,
then $ \lambda \not \in \sigma_p (\widehat A)$.
\item[2)]
If $\lambda \in \mathfrak{A}_p (\Sigma_+) \cap \mathfrak{A}_p (\Sigma_-)$,
then
\begin{description}
\item[(i)] $\lambda$ is an eigenvalue of $\widehat A$; the geometric multiplicity of $\lambda$ equals 1;
\item[(ii)] the eigenvalue $\lambda$ is simple if an only if at least one of conditions (\ref{e S-S=S-S}), (\ref{e Int+<Inf}),
(\ref{e Int-<Inf}) is not fulfilled;
\item[(iii)] if conditions \eqref{e S-S=S-S}, \eqref{e Int+<Inf} and
\eqref{e Int-<Inf} hold true, then the algebraic multiplicity of $\lambda$ equals
the greatest number $k$
($k \in \{2,3,4, \dots \} \cup \{ +\infty \} $) such that conditions (\ref{e frj<inf}) and
\begin{gather}  \label{e Gamma1j=Gamma1j}
 \lim\limits_{\substack{\ep \to 0 \\ \ep \in \R}} \MpmMm^{(j-2)}  (\lambda+i\ep) = 0 \quad
 \text{(the function} \  \MpmMm  \ \text{is defined in Lemma \ref{l eq cond})},
 \end{gather}
are fulfilled for all $ j \in \N$ such that $ 2 \leq j \leq k-1 $
(in particular, $k=2$ if at least one of conditions \eqref{e frj<inf}, \eqref{e Gamma1j=Gamma1j} is not fulfilled for $j=2$).
\end{description}

\item[3)]
Assume that
$\lambda \in \mathfrak{A}_r (\Sigma_+) \cap \mathfrak{A}_r (\Sigma_-)$.
Then $\lambda \in \sigma_p (\widehat A)$ if and only if
\begin{gather} \label{e fr1=fr1 Gc}
\lim\limits_{\substack{\ep \to 0 \\ \ep \in \R}} \MpmMm (\lambda+i\ep) = 0.
\end{gather}
If \eqref{e fr1=fr1 Gc} holds true, then the geometric multiplicity of $\lambda$
is 1, and the algebraic multiplicity is the greatest number $k$
($1 \leq k \leq \infty$) such that
the conditions (\ref{e frj<inf r}) and
\begin{gather}
\lim\limits_{\substack{\ep \to 0 \\ \ep \in \R}} \MpmMm^{(j-1)} (\lambda+i\ep) = 0 \label{e frj=frj r Gc}
\end{gather}
are fulfilled for all $ j \in \N$ such that $ 1 \leq j \leq k $.
\item[4.]
If $\lambda \in \mathfrak{A}_p (\Sigma_+) \cap \mathfrak{A}_r (\Sigma_-)$
or $\lambda \in \mathfrak{A}_p (\Sigma_-) \cap \mathfrak{A}_r (\Sigma_+)$,
then $\lambda \not \in \sigma_p (\widehat A)$.
\end{description}
\end{thm}

\begin{proof}
The proof is similar to that of Theorem \ref{t s p}, but some technical complications appear.
Namely, (\ref{e whG1}) is not valid whenever any of conditions in \eqref{e CpmCond} is not satisfied.

  We have to use (\ref{e whG1gen}), which is valid in the general case. Note that (\ref{e Gev12}) holds true.
   In the case $\la \in \mathfrak{A}_r (\Sigma_\pm)$, (\ref{e whG0}) holds also. When $\la \in \sigma_p (Q_{\Sigma_\pm})$, Eq.~(\ref{e whG0}) should be changed to \eqref{e chi1}.

The proof of statements \textbf{1)} and \textbf{4)} remains the same.

\textbf{2)} \ Let
$\lambda \in \mathfrak{A}_p (\Sigma_+) \cap \mathfrak{A}_p (\Sigma_-)$.
As before, we see that $\lambda$ is an eigenvalue of $\wh A$ with
 geometric multiplicity 1 and one of corresponding eigenvectors has the form (\ref{e y0 c2}).

Let
$ y_1 = \left( \begin{array}{c}
 y_1^- \\ y_1^+
\end{array} \right)$
and
$\wh A y_1 - \la  y_1 = y_0 $.
In the same way, we get \eqref{e G0=c=S}, \eqref{e y1p} as well as the fact that the conditions $y_1^+ \in L^2 (\R, d\Sigma_+ )$ and
$y_1^- \in L^2 (\R, d\Sigma_- )$ are
equivalent to (\ref{e Int+<Inf}) and (\ref{e Int-<Inf}),
respectively.
If \eqref{e Int+<Inf} and \eqref{e Int-<Inf}
are fulfilled, we obtain
\begin{gather*} 
\Gamma_1^\pm y_1^\pm = - \frac 1{ d\Sigma_\pm (\{\la\}) } \
\Gamma_1^\pm \frac {\chi_{ \R \setminus \{\la \}} (t)}{ t-\la }+
c_2^\pm  d\Sigma_\pm (\{\la\}).
\end{gather*}
The latter and \eqref{e G0=c=S} implies that $y_1 \in \dom (\wh A)$
if and only if conditions
(\ref{e S-S=S-S}) and
\begin{gather} \textstyle
- \frac 1{ d\Sigma_- (\{\la\}) }
\Gamma_1^- \frac {\chi_{ \R \setminus \{\la \}} (t)}{ t-\la } +
c_2^-  \ d\Sigma_- (\{\la \})
= - \frac 1{ d\Sigma_+ (\{\la \}) }
\Gamma_1^+ \frac {\chi_{ \R \setminus \{\la \}} (t)}{ t-\la } +
c_2^+  \ d\Sigma_+ (\{\la \}) \label{e hBC1 p gen}
\end{gather}
are fulfilled.
Thus, generalized eigenvectors of first order exist
if and only if conditions \eqref{e S-S=S-S}, \eqref{e Int+<Inf},
and \eqref{e Int-<Inf} are satisfied.
In this case, $ y_1$ has the form (\ref{e y1p}) with constants
$c_2^\pm $ such that (\ref{e hBC1 p gen}) holds. In particular, the constants
\begin{equation} \label{e c2 gen}
c_2^\pm = -\alpha_1^2
\Gamma_1^\mp \frac {\chi_{ \R \setminus \{\la \}} (t)}{ t-\la },
\end{equation}
give a generalized eigenvector (as before, $\alpha_1= \frac 1{d\Sigma_- (\{\la \})} =
\frac 1{d\Sigma_+ (\{ \la \}) } $).

Let $ y_2  = \left( \begin{array}{c}
y_2^- \\ y_2^+
\end{array} \right)$ and
$\wh A y_2 -\la y_2 = y_1$. Then (\ref{t-la y2}), (\ref{e t=la2 p}), and (\ref{e y2p})
have to be fulfilled with $c_2^\pm$ given by (\ref{e c2 gen}).
So $y_2^\pm$ belong to $
L^2 (\R, d\Sigma_\pm )$ if and only if (\ref{e frj<inf}) is satisfied for $j=2$. Conditions (\ref{e hBC})
are equivalent to $c_2^-=c_2^+$ and
\begin{gather*} 
-\alpha_1 \Gamma_1^- \frac { \chi_{ \R \setminus \{\la \}} } {(t-\la)^2}  -
c_2^- \Gamma_1^- \frac { \chi_{ \R \setminus \{\la \}} } {t-\la}  +
c_3^-  \alpha_1^{-1}
= \\ =
-\alpha_1 \Gamma_1^+ \frac { \chi_{ \R \setminus \{\la \}} } {(t-\la)^2}  -
c_2^+ \Gamma_1^+ \frac { \chi_{ \R \setminus \{\la \}} } {t-\la} +
c_3^+  \alpha_1^{-1} .
\end{gather*}
Thus $y_2$ exists if and only if, for $j=2$, conditions (\ref{e frj<inf}) and
\begin{gather}  \label{e Gamma1j=Gamma1j gc}
\Gamma_1^- \frac { \chi_{\R \setminus \{\lambda\} } } {(t-\lambda )^{j-1}}  =
\Gamma_1^+ \frac { \chi_{\R \setminus \{\lambda\} } } {(t-\lambda )^{j-1}}
\end{gather}
are fulfilled. By Lemma \ref{l eq cond}, (\ref{e Gamma1j=Gamma1j gc}) is equivalent to (\ref{e Gamma1j=Gamma1j}) with $j=2$.
Continuing this line of reasoning, we obtain parts \textbf{2)} and \textbf{3)} of the theorem.
%
%
\end{proof}

\begin{rem} \label{r Sterms}
\textbf{(1)}
In the last theorem, the conditions that determine the algebraic multiplicities are given in the terms of the function
$\Phi = M_{\Sigma_+, C_+} - M_{\Sigma_-, C_-} $, so in the terms of abstract Weyl functions $ M_{\Sigma_\pm, C_\pm} $.
Using Lemma \ref{l eq cond} and (\ref{e G1pm 1}), (\ref{e G1pm k}), Theorem \ref{t s p GC} can be easily rewritten in terms of the spectral measures $\Sigma_\pm$,
but this makes the answer longer due to the different forms of $\Gamma_1^\pm \ \frac { \chi_{\R \setminus \{\lambda\} } (t)} {(t-\lambda )^{j}}$ for the cases $j=1$ and $j \geq 2$, see \eqref{e G1pm 1} and \eqref{e G1pm k}. In the case when assumptions \eqref{e CpmCond} are fulfilled, \eqref{e G1pm 1} can be written in the form of  \eqref{e G1pm k} and we get Theorem \ref{t s p}.

\textbf{(2)} Note that eigenvalues of $\wh A$ that belong to $\rho (Q_{\Sigma_+} \oplus Q_{\Sigma_+})$ can be found
in the terms of $M_{\Sigma_\pm, C_\pm}$ using  \cite{DM91} (and, perhaps, \cite{D99}), see the next section.
Algebraic multiplicities of eigenvalues in $\rho (Q_{\Sigma_+} \oplus Q_{\Sigma_+})$ can be found using Krein's resolvent formula (see \cite{DM91,DM95} for a convenient abstract form), root subspaces for eigenvalues in $\rho (Q_{\Sigma_+} \oplus Q_{\Sigma_+})$ were found in \cite{Dth}. Theorem \ref{t s p GC} has some common points with \cite{BMN02},
where the abstract Weyl function was used to find eigenvalues of a self-adjoint operator. But the approach of the present paper goes in the backward direction:
we use the spectral measures $d\Sigma_\pm$ and the functional model to find eigenvalues and root subspaces and then, using Lemma \ref{l eq cond}, return to the answer in the terms of the abstract Weyl functions given in Theorem \ref{t s p GC}.

\textbf{(3)} Various generalizations of (R)-functions
and their functional models were considered in \cite{Fl00,J00}. These results were applied to certain classes of regular Sturm-Liouviile problems in \cite{FHdSW05,BT06,FHdSW09}.
\end{rem}

\subsection{Essential and discrete spectra of the model operator and of indefinite Sturm-Liouville operators}


Besides the symmetry condition $\sigma (T) = \sigma (T)^*$ the spectrum of a 
J-self-adjoint operator can be fairly arbitrary (see \cite{Lan82}).
An example of a differential operator with a "wild" spectrum was given in \cite{AM87,AM89}.

\begin{ex} \label{e AM}
Consider the operator $A$ in $L^2 [-1,1] $ associated with the differential expression 
\[
(\sgn x)\left( (\sgn x) y' \right)^\prime
\]
and boundary conditions $y(-1)=0=y(1)$. More precisely,
$A y = - y^{\prime\prime}$,
\begin{multline*}
\dom (A) = \{ y \in W_2^2 (-1,0) \oplus
W_2^2 (0,1) \ : \\ y(-0) = y(+0) , \ y^{\prime} (-0) = - y^{\prime} (+0) \quad
\text{and} \quad y(-1)=0=y(1) \} .
\end{multline*}
The operator $A$ is $J$-self-adjoint with $J$ given by $Jf(x) = (\sgn x)f(x)$.
It was observed in \cite{AM87,AM89} that every $\lambda
\in \C$ is an  eigenvalue of $A$ and, moreover, every $\lambda
\in \R$ is a nonsimple eigenvalue. Theorem \ref{t s p} shows that
every $\lambda \in \C$ is an eigenvalue  of infinite algebraic multiplicity (the geometric multiplicity of $\lambda $ equals 1).
Indeed, introducing as in Theorem \ref{l G+=G-} the operator $A_{\min} := A \cup A^*$,
we see that
\[
\dom (A) = \{ y \in A_{\min}^* \ : \ \Gamma_0^+ y_+ =\Gamma_0^- y_-, \ \Gamma_1^+ y_+ =\Gamma_1^- y_- \} ,
\]
where $y_+ (y_-)$ is the orthoprojection of $y$ on $L^2 [0,1]$ (resp., $L^2 [-1,0]$), 
\[
\Gamma_0^+ y_+ := -y'(+0), \quad \Gamma_0^- y_- := y'(-0), \quad \text{and} \ 
\Gamma_1^\pm y_\pm := y (\pm 0).
\]
 On the other hand,
$\{ \C, \Gamma_0^- , \Gamma_1^+ \}$ is a boundary triple for $A_{\min}^+ = A_{\min} \uph L^2 [0,1]$ and $\{ \C, \Gamma_0^- , \Gamma_1^- \}$
is a boundary triple for $A_{\min}^- = A_{\min} \uph L^2 [-1,0]$.
It is easy to see that the differential expression $-\frac {d^2}{dx^2}$ is associated with both the symmetric operators $A_{\min}^\pm$.
These operators and their boundary triples are unitarily equivalent. This means that the corresponding Weyl functions $M_\pm$ coincide. Now Theorem
\ref{t s p GC} implies that any $ \lambda \in \C \setminus \R$ is an eigenvalue of infinite algebraic multiplicity and therefore $\sigma (A) = \C$.
(Actually in this case conditions (\ref{e CpmCond}) hold, so Theorem \ref{t s p} can also be applied.) Finally, note that the functions $M_\pm$ are meromorphic and
therefore Theorem \ref{t s p GC} (2)-(3) and Lemma \ref{l eq cond} (ii)$\Leftrightarrow$(iii) imply that each point $\lambda \in \C$ is an eigenvalue of infinite algebraic multiplicity.
\end{ex}

\begin{rem} \label{r M}
In \cite{Ming04}, a characterization  of the case $\sigma (A) = \C$ was given in terms of coefficients for  regular operators $A=\frac {1}{r(x)} \frac {d}{dx} p(x)\frac {d}{dx} $ with Dirichlet boundary conditions.
Both coefficients $r$ and $p$ were allowed to change sign, modifications of arguments for general regular problems were suggested also.
\end{rem}

Arguments of Example \ref{e AM} show that the case $\sigma (\widehat A) = \C$ is exceptional in the sense of the next proposition.

\begin{prop} \label{p s=C}
The following statements are equivalent:
\begin{description}
\item[(i)] $\ M_{\Sigma_+,C_+} (\lambda ) =  M_{ \Sigma_-, C_-} (\lambda )$ for all $\ \lambda \in \C \setminus (\supp d\Sigma_+ \cup \supp d\Sigma_-) $;
\item[(ii)] the measures $d\Sigma_+$ and $d\Sigma_-$ coincide, and $C_+ = C_-$;
\item[(iii)] $\sigma (\widehat A) = \C$.
\end{description}
Moreover, if statements (i)-(iii) hold true, then every point in the set \\
$
\C \setminus \left( \sigma_{ess} (Q_{\Sigma_+}) \cup
\sigma_{ess} (Q_{\Sigma_-}) \right) $
is an eigenvalue of $\widehat A$ of infinite algebraic multiplicity.
\end{prop}

If $\ M_{\Sigma_+,C_+} (\cdot ) \not \equiv  M_{ \Sigma_-, C_-} (\cdot )$, then the nonreal spectrum is the set of zeros of analytic function  $\Phi$ defined in Lemma \ref{l eq cond}. More precisely, Theorem \ref{t s p GC} shows
 that
\begin{multline}
\sigma (\widehat A) \cap \rho ( Q_{\Sigma_+} \oplus Q_{\Sigma_-})
= \\ =
\{ \lambda \in \rho ( Q_{\Sigma_+}) \cap \rho ( Q_{\Sigma_-}) \ : \
M_{ \Sigma_+, C_+} (\lambda) = M_{\Sigma_-, C_-} (\lambda) \}
\subset \sigma_p (\widehat A) \label{e s in rho}
\end{multline}
(this statement also can be obtained from \cite[Proposition 2.1]{DM91}).
It is easy to see that \eqref{e s in rho} and Theorem \ref{t s p GC}
yield the following description of the discrete and essential spectra
(cf. \cite[p.~106, Theorem~1]{AhGl2}).

\begin{prop} \label{p e and d}
Assume that $M_{\Sigma_+,C_+} (\lambda_0) \neq M_{ \Sigma_-, C_-} (\lambda_0)$ for certain
$\lambda_0$ in the set $\rho ( Q_{\Sigma_+}) \cap \rho ( Q_{\Sigma_-}) $.
Then:
\begin{description}
\item[(i)] $ \sigma_{\ess} (\wh A) = \sigma_{\ess}
(Q_{\Sigma_{+}})
\cup \sigma_{\ess} (Q_{\Sigma_{-}}) \subset \R;$
\item[(ii)] 
\begin{multline*}
 \qquad \sigma_{\disc} (\wh A) = \left( \sigma_{\disc}
(Q_{\Sigma_{+}}) \cap \sigma_{\disc} (Q_{\Sigma_{-}}) \right) \
\cup \\
  \cup \
\{ \lambda \in
\rho (Q_{\Sigma_+}) \cap \rho (Q_{\Sigma_-})
\ : \ M_{\Sigma_+,C_+} (\lambda ) = M_{ \Sigma_-, C_-} (\lambda ) \} ;
\end{multline*}
\item[(iii)] the geometric multiplicity equals 1 for all
eigenvalues of $\wh A$;
\item[(iv)] if $\lambda_0 \in  \left( \sigma_{\disc}
(Q_{\Sigma_{+}}) \cap \sigma_{\disc} (Q_{\Sigma_{-}}) \right)$,
then the algebraic multiplicity of $\lambda_0$
is equal to the multiplicity of $\lambda_0$ as a zero
of the holomorphic function \\
$ \frac 1{M_{\Sigma_+,C_+} (\lambda)} - \frac 1{M_{\Sigma_-,C_-} (\lambda)}$;
\item[(v)] if $\lambda_0 \in \rho (Q_{\Sigma_{+}}) \cap
\rho (Q_{\Sigma_{-}})$, then the algebraic multiplicity
of $\lambda_0$ is equal to the multiplicity of
$\lambda_0$ as zero of the holomorphic function
$M_{\Sigma_+,C_+} (\lambda) - M_{\Sigma_-,C_-} (\lambda)$.
\end{description}
\end{prop}

\begin{rem} \label{r defin}
The operator $\wh A$ is definitizable if and only if the sets $\supp d\Sigma_+ $ and $\supp d\Sigma_-$ are separated by a finite number of points
(in the sense of \cite[Definition 3.4]{KarTr07}). This criterion was obtained for operators $A = (\sgn x)(-d^2/dx^2 +q(x))$  in \cite{KarDis,KarKr05}
(see also \cite[Section 2.3]{KarMMM07})
using the result of \cite{JL79} and the fact that $\rho (A) \neq \emptyset$; the detailed proof was published in \cite[Theorem 3.6]{KarTr07}.
The same proof is valid for the operator $\wh A$ if we note that $\rho (\wh A) \neq \emptyset$ whenever
$\supp d\Sigma_+ $ and $\supp d\Sigma_-$ are separated by a finite number of points. Indeed, in this case $\supp d\Sigma_+ \neq \supp d\Sigma_-$  since
$\supp d\Sigma_\pm $ are unbounded due to the second assumption in \eqref{e S assump}.
\end{rem}

\subsection{Non-emptiness of resolvent set for Sturm-Liouville operators \label{ss nonempty}}

To apply Proposition \ref{p e and d} to the J-self-adjoint Sturm-Liouville operator
\[
A = \frac {\sgn x}{|r(x)|} \left( -\frac {d}{dx}p(x) \frac {d}{dx} +q(x) \right)
\]
introduced in Section \ref{ss SLModel}, one has to insure that $\rho (A) \neq \emptyset$.
Here we discuss briefly results of this type. We assume that $\wh A \{ \Sigma_+,C_+, \Sigma_-,C_-\}$ is
one of model operators unitarily equivalent to the operator $A$ and that $M_\pm (\cdot) = M_{\Sigma_\pm,C_\pm} (\cdot)$ are the associated Weyl functions.

Sometimes it is known that the asymptotic formulae of $M_+ $ and $M_- $ at $\infty$ are different.
 This argument was used in \cite[Proposition 2.5 (iv)]{KarMMM07} to show that $\rho (A) \neq \emptyset$ for the operator $(\sgn x) (-d^2/dx^2 +q(x))$. Indeed,
\eqref{e asM} shows that $M_{\Sl+} (\cdot) \not \equiv M_{\Sl-} (\cdot)$. One can extend this result using \cite[Theorem 4]{At81} in the following way:
\emph{if $p \equiv 1$ and there exist constants $r_\pm >0$ such that
\[
\int_0^x ( \pm r(t) - r_\pm) dt = o(x) \qquad \text{as} \quad x \to \pm 0 ,
\]
then $\rho (A) \neq \emptyset$.}

If $p \not \equiv 1$, one may use the standard change of variable $s = \int_0^x \frac{d\tau}{p(\tau)}$ to get back to the form with $p \equiv 1$:


\begin{prop} \label{p nemp as gen}
Assume that there exist positive constants $r_\pm$ such that
\begin{gather}
\int_0^x r(t) dt = \left(  r_+  \, + \, o(1)  \right) \ \int_0^x \frac{dt}{p(t)} \  \qquad \text{as} \quad x \to +0 , \\
\int_x^0 |r(t)| dt = \left( r_- \, + \, o (1) \right) \ \int_x^0 \frac{dt}{p(t)}   \qquad \text{as} \quad x \to -0 .
\end{gather}
Then $\rho (A) \neq \emptyset$. 
\end{prop}

Another simple way to prove $\rho (A) \neq \emptyset$ uses information on the supports of spectral measures $d\Sigma_\pm$. In this way,
it was obtained in \cite[Proposition 3.1]{KarKos08} that $\rho(A) \neq \emptyset$ if $L= \frac 1{|r|} (-\frac d{dx} p \frac d{dx} +q)$ is semi-bounded from below
(the proof \cite[p. 811]{KarKos08} given for $p \equiv 1$ is valid in the general case). Moreover,
modifying slightly the same arguments, we get the next result.

\begin{prop} \label{p nemp sb}
Assume that at least one of the symmetric operators $A_{\min}^+$, $A_{\min}^-$ (defined in Section \ref{ss SLModel}) is semi-bounded.
Then $\rho(A) \neq \emptyset $.
\end{prop}


\begin{rem} \label{r loc def}
\textbf{(1)}
 Proposition \ref{p nemp sb} has the following application to the theory of locally definitizable operators (see \cite{J03} for basic definitions): \emph{the operator $A = \frac {\sgn (x)}{|r|} (-\frac d{dx} p \frac d{dx} +q)$ introduced in Section \ref{ss SLModel} is locally definitizable in some open neighborhood of $\infty$ if and only if corresponding operator $L = \frac 1{|r|} (-\frac d{dx} p \frac d{dx} +q)$ is semi-bounded from below.}
This is a natural generalization of \cite[Theorem 3.10]{KarTr07}, where the above criterion for $r(x) = \sgn x$ and $p \equiv 1$ was obtained.
The proof of \cite[Theorem 3.10]{KarTr07} (based on \cite{B06}) remains valid in general case if Proposition \ref{p nemp sb}  is used instead of \cite[Proposition 2.5 (iv)]{KarMMM07}.

Local definitizability of Sturm-Liouville operators 
with the weight function $r$ having more than one turning point was considered in \cite{B07}. 

\textbf{(2)} And vise versa, it was noticed in \cite[Proposition 4.1]{KarTr07} that local definitizability results could be used to get additional information on non-real spectrum. Namely, the above criterion of local definitizability implies that \emph{the non-real spectrum $\sigma(A) \setminus \R$ of the
operator $A = \frac {\sgn (x)}{|r|} (-\frac d{dx} p \frac d{dx} +q)$ is bounded if the operator $L = \frac 1{|r|} (-\frac d{dx} p \frac d{dx} +q)$ is semi-bounded from below} (the proof is immediate from the definition of the local definitizability).

\textbf{(3)} Under the assumption that $\aexp [y] = (\sgn x)(-y''  +qy)$ is in the limit point case in $\pm \infty$, the fact that $\rho (A) \neq \emptyset$ was noticed by
M.~M.~ Malamud and the author of this paper during the work on \cite{KarMMM04}, and was published in \cite{KarKr05, KarMMM07}. 
\end{rem}

\section{The absence of embedded eigenvalues and other applications}
\label{s apl}

\subsection{The absence of embedded eigenvalues for the case of infinite-zone potentials}

Theorems \ref{t s p} and \ref{t s p GC} can be applied to prove that the  Sturm-Liouville operator $A$ has no embedded eigenvalues in the essential spectrum
if some information on the spectral measures $d\Sigma_\pm$ is known.

We illustrate the use of this idea on operators $A = (\sgn x)L$, where $L=-d^2/dx^2 + q(x)$ is an operator in $L^2 (\R)$ with infinite-zone potentials $q$ (in the sense of \cite{Lev84}, the definition is given below).
 First recall that the operator $L=-d^2/dx^2 + q(x)$ with infinite zone potentials $q$ is defined on the maximal natural domain and is self-adjoint in $L^2 (\R)$ (i.e., the differential expression is in the limit point case both at $\pm\infty$). The spectrum of $L$ is absolutely
 continuous and has the zone structure, i.e.,
\begin{equation} \label{e sinf}
\sigma (L) = \sigma_{\ac} (L)=[\mur_0 , \mul_1 ] \cup [\mur_1 ,
\mul_2 ] \cup \cdots,
\end{equation}
where $\{ \mur_j \}_0^\infty$ and $\{ \mul_j \}_{j=1}^{\infty}$
are sequences of real numbers such that
\begin{equation} \label{e mu<mu}
\mur_0 < \mul_1 < \mur_1 < \dots < \mur_{j-1} < \mul_j < \mur_j <
\dots \quad ,
\end{equation}
and
\[
\lim_{j\to \infty} \mur_j = \lim_{j\to \infty} \mul_j = +\infty.
\]
$\mul_j$ ($\mur_j$) is the left (right, resp.) endpoint of the j-th gap in the spectrum $\sigma (L)$,
the "zeroth" gap is $(-\infty, \mur_0$).
Following \cite{Lev84}, we briefly recall the definition of infinite-zone potential under the additional assumptions that
 \begin{equation} \label{e SumGap}
\sum_{j=1}^\infty \mur_j (\mur_j - \mul_j) < \infty, \qquad
\sum_{j=1}^\infty \frac 1{\mul_j} < \infty . \quad
\end{equation}
Consider infinite sequences  $\{ \xi_j \}_1^\infty $ and
$\{\epsilon_j \}_1^\infty $ such that $ \xi_j \in [\mul_j , \mur_j
] $, $\epsilon_j \in \{-1,+1\}$
 for all $j \geq 1$.
For every $N \in \N$, put
\begin{eqnarray}
  &
g_N  =  \prod_{j=1}^N \frac{\xi_j - \lambda }{\mul_j}, \quad
f_N  = (\lambda - \mur_0) \prod_{j=1}^N \frac{\lambda - \mul_j}
{\mul_j}
\frac{\lambda- \mur_j }{\mul_j}, \label{e fNgN} \\
 &
k_N (\lambda)  =  g_N (\lambda)
\sum_{j=1}^N \frac{\epsilon_j \sqrt{-f_N
(\xi_j)}}{g'_N(\xi_j)(\lambda-\xi_j)} , \quad h_N (\lambda)  =
\frac{f_N (\lambda) + k_N^2 (\lambda)}{g_N (\lambda)}.
\label{e kNhN}
\end{eqnarray}

It is easy to see from (\ref{e SumGap}) that $g_N$ and $f_N$
converge uniformly on every compact subset of $\C$. Denote 
\[
\lim_{N \to \infty} g_N (\lambda) =: g (\lambda), 
\quad \lim_{N \to
\infty} f_N (\lambda) =: f (\lambda).
\]
 \cite[Theorem 9.1.1]{Lev84}
states that there exist limits 
\[
\lim_{N \to \infty} h_N (\lambda)
=: h (\lambda), \quad \lim_{N \to \infty} k_N (\lambda) =: k
(\lambda) \quad \text{for all} \ \ \lambda \in \C.
\]
 Moreover, the functions $g$,
$f$, $h$, and $k$ are holomorphic in $\C$.

It follows from \cite[Subsection 9.1.2]{Lev84} that
 the functions
     \begin{equation}\label{e ML inf}
m_{\Sl\pm}(\lambda)  :=  \pm  \frac {g(\lambda)}{ k(\lambda) \mp 
i\sqrt{f( \lambda)}}
       \end{equation}
are the Titchmarsh-Weyl $m$-coefficients on  $\R_{\pm}$
(corresponding to the Neumann boundary conditions) for some
Sturm-Liouville operator  $L=-d^2/d x^2 +q(x)$ with a real bounded
potential $q(\cdot)$. The branch $\sqrt{f( \cdot)}$ of the
multifunction is  chosen  such  that both  $m_{\pm}$ belong
to the class $(R)$ (see Section \ref{ss SLModel} for the definition).

          \begin{defn}[\cite{Lev84}] \label{d infZp}
A real potential $q$ is called an infinite-zone potential if the
Titchmarsh-Weyl $m$-coefficients $m_{\Sl\pm}$ associated with $-d^2/d x^2 +q(x)$ on $\R_{\pm}$ admit representations
(\ref{e ML inf}).
       \end{defn}

Let $q$ be an infinite-zone potential defined as above. B.~Levitan proved that under
the additional condition $\inf (\mul_{j+1} - \mul_j) >0$, the potential $q$ is \emph{almost-periodical} (see \cite[Chapter 11]{Lev84}).

The following theorem describes the structure of the spectrum of the J-self-adjoint operator $A = (\sgn x )L$. Note that the Titchmarsh-Weyl m-coefficients $M_{\Sl\pm}$ for $A$ introduced in Section \ref{ss SLModel} are connected with m-coefficients for $L$ through
\begin{equation} \label{e M=m}
M_{\Sl\pm} (\lambda) = \pm m_{\Sl\pm} (\pm \lambda),
\quad \lambda \in \C \setminus \supp d\Sigma_\pm \quad \text{(see e.g. \cite[Section 2.2]{KarMMM07})}.
\end{equation}

\begin{thm}\label{th_InfZ}
Let $L=-d^2/d x^2 + q(x)$ be a Sturm-Liouville operator with an
infinite-zone potential $q$ and let $A=(\sgn x)L$. Assume also that assumptions (\ref{e SumGap}) are satisfied for the zones of the spectrum
$\sigma(L)$. Then:
 \begin{description}
 \item[(i)] $\sigma_{\mathrm{p}}(A)=\sigma_{\disc} (A)$, that is all the eigenvalues of $A$ are isolated and have finite algebraic multiplicity.
 Besides, all the eigenvalues and their geometric and algebraic multiplicities are given by statements (ii)-(v) of Proposition \ref{p e and d}.
 \item[(ii)] The nonreal spectrum $\sigma(A) \setminus \R $ consists of a finite number of eigenvalues.
 \item[(iii)] $\sigma_{\ess} (A) = \left( \bigcup_{j=0}^\infty [\mu^r_j, \mu^l_{j+1} ] \right) \
 \cup \ \left( \bigcup_{j=0}^\infty [- \mu^l_{j+1}, -\mu^r_j  ] \right) $.
 \end{description}
\end{thm}

The functions $g$, $f$, $k$, and $h$ defined above are holomorphic
in $\C$. Moreover, $g$ and $f$ admit the
following representations
\[ 
g (\lambda)=  \prod_{j=1}^\infty \frac{\xi_j - \lambda }{\mul_j},
\quad f (\lambda)= (\lambda- \mur_0) \prod_{j=1}^N \frac{\lambda - \mul_j}
{\mul_j} \frac{\lambda- \mur_j }{\mul_j}, 
\]  
where the infinite products converge uniformly on all compact
subsets of $\C$ due to assumptions (\ref{e SumGap}) (see
\cite[Section 9]{Lev84}).
It follows from (\ref{e kNhN}) that
\begin{equation} \label{e hg-k2=f}
h_N (\lambda) g_N (\lambda) - k^2_N (\lambda) = f_N (\lambda)
\quad \mathrm{and} \quad h (\lambda) g (\lambda) - k^2
(\lambda) = f (\lambda).
\end{equation}
From this and (\ref{e M=m}) we get
       \begin{equation}\label{e M infZ}
M_\pm (\lambda ) = \frac {g(\pm \lambda)}{ k(\pm \lambda) \mp
i \sqrt{f(\pm \lambda)}} =  \frac {k(\pm \lambda) \pm 
i\sqrt{f(\pm \lambda)}  }{ h (\pm \lambda)} .
      \end{equation}
It follows from (\ref{e hg-k2=f}) that the zeros of the function $h$ belongs to $(-\infty, \mu^r_0] \cup \bigcup_{j=1}^\infty [\mu^l_j, \mu^r_j] )$.
Besides, all the zeroes of $h$ have multiplicity 1 (otherwise one of the functions $M_{\Sl\pm}$ does not belongs to the class (R)).
This implies that the spectra of the operators $A_0^\pm$ defined in Section \ref{ss SLModel} have the following structure:
\begin{gather}
\sigma_{\ess} (A_0^\pm ) =  \sigma_{\ac} (A_0^\pm ),
\qquad \sigma_{\ac} (A_0^+) = - \sigma_{\ac} (A_0^-) = \sigma (L) = \bigcup_{j=0}^\infty [\mu^r_j, \mu^l_{j+1} ] ,
\label{e sacA2pm} \\
\sigma_p (A_0^\pm) = \sigma_{\disc} (A_0^\pm) = \{ \lambda \in \R \setminus \sigma (\pm L) : h (\pm \lambda) = 0 ,
 \ \  k(\pm \lambda) \pm \mathrm{i} \sqrt{f(\pm
\lambda)} \neq 0 \}.\label{e SpDisc FZA2}
     \end{gather}

\begin{proof}[Proof of Theorem \ref{th_InfZ}]
Since $\sigma (A) \neq \C$ (see Section \ref{ss nonempty}), Proposition \ref{p e and d} (i) proves \textbf{(iii)}.

\textbf{(i)} We have to show that there are no eigenvalues in $\sigma_{\ess} (A) $.
This statement follows from the fact that
$ \pm \bigcup_{j=0}^\infty [\mu^r_j, \mu^l_{j+1} ] \subset \mathfrak{A}_0 (\Sigma_{\Sl\pm}) $ (see \eqref{e fA0} for the definition) and Theorem \ref{t s p} (1).
Indeed, let $\lambda_0$ be in $\bigcup_{j=0}^\infty [\mu^r_j, \mu^l_{j+1} ]$. Note that \eqref{e sacA2pm} implies that
$\lambda_0 \in \sigma_{ac} (A_0^+) = \sigma_{ac} (Q_{\Sigma_{\Sl+}}) $.
It follows from \eqref{e M infZ}, \eqref{e MpmInt} and the Stieltjes inversion formula (see e.g. \cite[formula (2.9)]{KarKos08}) that
\begin{gather}
\Sigma_{\Sl\pm} ^\prime (t) = \frac{\sqrt{ f (\pm t)}} { \pi h (\pm t) }
\quad \text{for a.a.} \quad t \in \pm \bigcup_{j=0}^\infty [\mu^r_j, \mu^l_{j+1} ]
          \label{e Sac pr}
\end{gather}
In particular, if $\lambda_0 \in \bigcup_{j=0}^\infty (\mu^r_j, \mu^l_{j+1} )$,
then $\Sigma_{\Sl+}' (t) \geq C_1 $ for $|t-\lambda_0|$ small enough and a certain constant $C_1>0$. So
\begin{equation} \label{e int l0}
 \int_\R |t-\lambda_0|^{-2} d\Sigma_{\Sl+} (t) = \infty .
 \end{equation}
In the case when $\lambda_0 = \mu^r_j$ (or $\lambda_0 = \mu^l_j$), \eqref{e int l0} also holds since
$\Sigma_+' (t) \geq C_2 |t - \lambda_0 |^{1/2}$, $C_2 >0$, for $t$ in a certain left (resp., right) neighborhood of $\lambda_0$.
Arguments for $ \sigma (-L) \subset \mathfrak{A}_0 (\Sigma_-) $ are the same. This concludes the proof of \textbf{(i)}.

\textbf{(ii)} We have to prove only that $\sigma (A) \setminus \R$ is finite. \cite[Proposition 4.1]{KarTr07} (see also Remark \ref{r loc def} (2)) implies that
\begin{equation} \label{e sAbound}
\sigma (A) \setminus \R \quad \text{is a bounded set.}
\end{equation}

Proposition \ref{p e and d} (ii) states that points of $\sigma (A) \setminus \R$ are zeros of the holomorphic in $\C \setminus (\supp d\Sigma_{\Sl+} \cup \supp d\Sigma_{\Sl-})$ function  $M_{\Sl+} ( \lambda) - M_{\Sl-} ( \lambda)$, and therefore, are roots of each of the following equations
(each subsequent equation is a modification of the previous one)
\[
\frac {g(\lambda)}{ k(\lambda) - 
i \sqrt{f( \lambda)}} = \frac {g(-\lambda)}{ k(-\lambda) + 
i \sqrt{f( -\lambda)}} ,
\]
\begin{multline*}
(g(\lambda)k(-\lambda) - g(-\lambda)k(\lambda))^2
= \\ =
- g^2(\lambda) f( -\lambda) - g^2 (-\lambda) f( \lambda) -
2g(\lambda) g(-\lambda) \sqrt{f( -\lambda) } \sqrt{f( \lambda)} ,
\end{multline*}
\begin{multline*}
\Bigl( [g(\lambda)k(-\lambda) - g(-\lambda)k(\lambda)]^2 + g^2(\lambda) f( -\lambda) + g^2 (-\lambda) f( \lambda) \Bigr)^2
- \\ -
4 g^2(\lambda) g^2(-\lambda) f( -\lambda) f( \lambda)= 0 .
\end{multline*}
The entire function in the left side of the last equation is not identically zero since it is positive in the points of the set
\[
\{ \lambda \in (\sigma (L) \setminus \sigma (-L)) \cup (\sigma (-L) \setminus \sigma (L)) \ : \ g(\lambda) g(-\lambda) \neq 0 \}
\]
(this set is nonempty due to \eqref{e sinf}, \eqref{e ML inf}, and the fact that $M_{\Sl\pm} (\cdot) \not \equiv 0$).
Therefore,
\begin{multline} \label{e sp inf-z}
M_{\Sl+} - M_{\Sl-} \quad \text{has a finite number of zeros} \\
\text{in any bounded subset of} \quad \C \setminus (\supp d\Sigma_{\Sl+} \cup \supp d\Sigma_{\Sl-}) \quad (\supset \C \setminus \R).
\end{multline}
Combining this and \eqref{e sAbound}, we see that $\sigma (A) \setminus \R$ is finite.
\end{proof}

Note that \eqref{e sp inf-z} implies that each gap of the essential spectrum $\sigma_{\ess} (A)$ has at most finite number of eigenvalues.


\subsection{Other applications}
\label{ss other appl}
A part of this paper was obtained in author's candidate thesis \cite{KarDis}, was announced together with some applications in the short communication \cite{KarKr05},
and was a base of several author's conference talks in 2004-2005.
Namely, in \cite{KarDis,KarKr05}, the operator $\wh A \{\Sigma_+, C_+, \Sigma_-, C_- \}$ was introduced as a functional model for the operator $(\sgn x) (-d^2/dx^2 +q)$ (see Theorem \ref{t SLModel})
and  the description of eigenvalues under conditions \eqref{e CpmCond} was given (see Theorem \ref{t s p}).
These results were used in \cite{KarMMM07,KarKos08}.
The idea of the functional model originated from \cite{KarKos03,KarMMM04}, where a representation of the operator $(\sgn x) (-d^2/dx^2 +q)$
as an extension of a direct sum of two symmetric Sturm-Liouville operators was used essentially
(in a less explicit form the same idea appeared earlier in \cite{CN95,FN98}).

The absence of embedded eigenvalues in the essential spectrum of the operator $(\sgn x) (-d^2/dx^2 +q)$
with a finite-zone potential $q$ was proved in \cite[Theorem 7.1 (2)]{KarMMM07} via Theorem \ref{t s p}.
This proof is adopted in part (i) of Theorem \ref{th_InfZ} for the infinite-zone case.

Theorem \ref{t s p} helps to find algebraic multiplicity
of embedded eigenvalues. This was used in a paper of A.~Kostenko and the author
(see \cite{KarKos08}, Proposition 2.2, Theorems 6.1 (ii) and 6.4 (ii)) to prove simplicity of the eigenvalue $\lambda = 0$
for two operators of
type $(\sgn x)(-d^2/dx^2 +q)$. 
This fact and necessary conditions for regularity of
critical points (see \cite[Theorem 3.9]{KarKos08}) allowed us to show that $0$ is a singular critical point for the considered operators
(see \cite[Remark 6.3, Theorem 6.4 (iii)]{KarKos08} and also Section \ref{s sing0} of the present paper).

The fact that $0$ may be a non-semi-simple eigenvalue (i.e., $\ker A^2 \neq \ker A$) is essential for
the theory of "two-way" diffusion equations. In the simplest case,
such equations lead to spectral analysis of $J$-nonnegative operators that take the form of the operator $A$ introduced in Section \ref{ss SLModel}. If $0$ and $\infty$ are not singular critical points of $A$ than
the algebraic multiplicity of $0$ affects proper settings of boundary value problems for the corresponding diffusion equation
(see \cite{BP83,B85,GvdMP87}).
If $0$ is a singular critical point of $A$,
as in the examples constructed in \cite{KarKos08,KarKosM09} and Section \ref{s sing0},
the existence and uniqueness theory for corresponding diffusion equations is not well-understood
(see \cite[Section 1]{KvdMP87}, \cite{C00,vdM08,Kar08}).

\begin{rem} \label{r embeig}
For periodic potentials and certain classes of decaying potentials, asymptotic behavior of solutions of
$-y''+qy=\lambda y $ is well-known and yields the absence of eigenvalues in some parts of the spectrum
of the operator $A=(\sgn x)L$ (where $L=-d^2/dx^2+q $ is a self-adjoint operator in $L^2 (\R)$).

For example, the assumption $\int_{-\infty}^{+\infty} (1+|x|) |q(x)| dx < \infty$ yields $\sigma_p (A) \cap \R = \emptyset$ (see e.g. \cite[Lemma 3.1.1 and formula (3.2.4)]{Mar77}). This fact was used essentially in \cite[Section 4]{KarKosM09} to prove that, for this class of potentials, $A$ is similar to self-adjoint operator exactly when $L\geq 0$. For periodic potentials, $\sigma_{\ess} (A) \cap \sigma_p (A) = \emptyset$ follows from
\cite[Section 21]{Tit58}. For $q\in L^1 (\R)$, the fact that the only possible real eigenvalue is $0$ (i.e., $\sigma_p (A) \cap \R \subset \{0\}$) follows immediately from \cite[Section 5.7]{Tit58} or from \cite[Problem IX.4]{CodLev58Rus}.
\end{rem}

\section{Remarks on indefinite Sturm-Liouville operators with the singular critical point $0$}
\label{s sing0}

In this section, we provide an alternative approach to examples of $J$-self-adjoint Sturm-Liouville operators
with the singular critical point $0$ constructed in \cite[Sections 5 and 6]{KarKos08} and \cite[Section 5]{KarKosM09}.

As before, let $H$ be a Hilbert space  with a scalar product $(\cdot, \cdot
)_H$. Suppose that $H = H_+  \oplus H_- $, where $H_+$ and $H_-$
are (closed) subspaces of $H$. Denote by $P_\pm$ the orthogonal
projections from $H$ onto $H_\pm$. Let $ J=P_+ - P_-$ and
$[ \cdot, \cdot ] := (J \cdot ,\cdot )_H $. Then the pair $K =
(H , [\cdot, \cdot])$ is called a \emph{Krein space}.
The operator $J$ is called a fundamental
symmetry (or a signature operator) in the Krein space $K$.
Basic facts on the theory of Krein spaces and the theory of J-self-adjoint definitizable operators
can be found e.g. in \cite{Lan82}.
An account on (non-differential) operators with a finite singular critical point can be found in \cite{CGL00}. The following proposition is
a simple consequence of \cite[Theorem II.5.7]{Lan82}.

\begin{prop} \label{p cr p 0}
Assume that a $J$-self-adjoint definitizable operator $B$ has a simple eigenvalue $\lambda_0 \in \R$
and that a corresponding eigenvector $f_0 \in \ker (B-\lambda_0) \setminus \{0\}$ is neutral (i.e., $[f_0,f_0]=0$).
Then $\lambda_0$ is a singular critical point of $B$.
\end{prop}

\begin{proof}

Assume that $\lambda_0$ is not a singular critical point of $B$. Then there exist a $J$-orthogonal decomposition into a direct sum 
 of two closed subspaces $H = H_0 [\dot +] H_1  $ that reduces the operator $B$ and such that $H_0$ is a root subspace corresponding to $\lambda_0$ and $\lambda_0 \not \in \sigma_p (B \uph H_1)$
(this follows from \cite{Lan82}, Proposition II.5.1, Proposition II.5.6, Theorem II.5.7, and the decomposition (II.2.10)).
Since the eigenvalue $\lambda_0$ is simple, we have $H_0 = \{ cf_0 , c\in \C \}$.
Since $f_0$ is neutral, we see that $[f_0,f] =0$ for all $f$ in $H$. So $\{ H, [\cdot,\cdot] \}$
is not a Krein space, a contradiction.
\end{proof}

It is not difficult to see that the operators
considered in \cite[Sections 5 and 6]{KarKos08} and \cite[Section 5]{KarKosM09} satisfy conditions of Proposition \ref{p cr p 0},
and so Proposition \ref{p cr p 0} can be used in these two papers instead of \cite[Theorem 3.4]{KarKos08}.
This does not makes the proofs much simpler, but the results become clearer from the Krein space point of view.

\begin{rem}
The necessary similarity condition given
in \cite[Theorem 3.4]{KarKos08} is of independent interest
since it provides a criterion of similarity to a self-adjoint operator for operators $(\sgn x)(-d^2/dx^2+q)$
with finite-zone potentials (see \cite[Remark 3.7]{KarKos08}).
And it is unknown \emph{whether the condition of \cite[Theorem 3.4]{KarKos08} provides a criterion of similarity to a self-adjoint operator
for the general operator $\frac {\sgn x}{|r|} (-\frac {d}{dx} p \frac {d}{dx} +q)$ with one turning point
introduced in Section \ref{ss SLModel}}.
\end{rem}

Using Proposition \ref{p cr p 0}, a large class of operators with the singular critical point $0$ similar to that of \cite{KarKos08,KarKosM09} can be constructed.
In the next theorem, we characterize the case described in Proposition \ref{p cr p 0} among the operators $\A_{r}:=-\frac {\sgn x}{|r(x)|} \frac {d^2}{dx^2}$  that have the limit point case both at $\pm \infty$ (and act in $L^2 (\R; |r(x)| dx)$). So we assume that
\begin{gather}
r \in L^1_{\loc} (\R),  \quad r(x) = (\sgn x) |r(x)| ,  \label{e rsgnr} \\
\int_{\R_\pm } x^2 |r(x)| dx = \infty , \label{e rinf}
\end{gather}
and the operator $\A_{r}$ is defined on its maximal domain. Condition \eqref{e rinf} is equivalent to $J$-self-adjointness of $\A_r$
with $J: f(x) \mapsto (\sgn x) f(x)$. Obviously, $\A_r$ is $J$-non-negative and definitizable (see \cite{CL89} and also Proposition \ref{p nemp sb}).

\begin{thm} \label{t cp0}
Let assumptions \eqref{e rsgnr}, \eqref{e rinf} be satisfied. Then:
\begin{description}
\item[(1)] Two following statements are equivalent:
\begin{description}
\item[\textbf{(r1)}] $\A_r$ has a simple eigenvalue at $0$ and $[f_0,f_0]=0$ for $f_0 \in \ker \A_r \setminus \{0\}.$
\item[\textbf{(r2)}] $r \in L^1 (\R)$, $\int_{\R} r(x) dx = 0$, and
\begin{gather} \label{e y1 norm}
\int_{\R} (y_1 (x))^2  |r(x)| dx = +\infty  , \quad \text{where} \quad y_1 (x) := \int_0^x \int_s^{+\infty} r(t) \ dt ds .
\end{gather}
\end{description}
\item[(2)] If \textbf{(r2)} is satisfied, then $0$ is a singular critical point of $\A_r$.
\end{description}
\end{thm}

\begin{proof}
We need to prove only \textbf{(1)}. Let us note that $r \in L^1 (\R) $ is equivalent to $0 \in \sigma_p (\A_r) $. If the latter holds, then $f_0 (x) \equiv 1$ is an eigenfunction (unique up to multiplication by a constant), and $[f_0,f_0]=0$ is equivalent to $\int_{\R} r(x) dx = 0$.

Assume that the eigenvalue $0$ is not simple. Then there is a generalized eigenfunction of first order $y \in L^2 (\R, |r(x)|dx)$, which is
a solution of $\A_r y = f_0$. It is easy to see that the derivative of $y$ has the form $y' = y_1^{\prime} + C_1$, where $C_1 \in \C$ is a constant
and $y_1$ is defined by \eqref{e y1 norm}.
Condition $y \in L^2 (\R, |r(x)|dx)$ implies that $C_1 = 0$ (otherwise $y'(x) \to C_1 \neq 0 $ as $x \to \pm \infty$
and, therefore, $y \not \in L^2 (\R; |r(x)| dx)$ due to \eqref{e rinf}).
This shows that $y_1$ is the only possible generalized eigenvector, and \eqref{e y1 norm} ensures that $y_1 \not \in L^2 (\R, |r(x)|dx)$.
Thus, \eqref{e y1 norm} is equivalent to the fact that $0$ is a simple eigenvalue (under the assumptions $r \in L^1 (\R)$, $\int_{\R} r(x) dx = 0$).
\end{proof}

In the following corollary $f(x) \asymp g (x)$ as $x \to +\infty \ (x \to -\infty)$ means that for $X>0$ large enough both $\frac fg$
and $\frac gf$ are bounded on $(X,+\infty)$ (resp.,  $(-\infty,-X)$).

\begin{cor}
Assume that condition \eqref{e rsgnr} is satisfied and $r (x)  \asymp  \pm |x|^{\alpha}$ as $x \to \pm \infty$.
If $\alpha \in [-5/3,-1) $ and $\int_{\R} r(x) dx = 0$, then the operator
$\A_r$ has a singular critical point at $0$.
\end{cor}

This includes \cite[Theorem 5.1 (ii)-(iV)]{KarKos08}, also it is interesting to compare this result with the sufficient condition on regularity of $0$
given in \cite{Kos06},\cite[Theorem 1.3]{KarKosM09}, and the discussions in \cite[Section 5.2]{KarKosM09} and \cite[Section 5]{KarKos09}.

\begin{rem} \label{r sing0}
Theorem \ref{t cp0} can be easily generalized to weight functions $r$ with many turning points under the assumption that
\begin{equation} \label{e c sign}
r(x) \quad \text{is of constant sign a.e. on} \quad (-\infty, -X) \quad \text{and} \quad (X,+\infty)
\end{equation}
for certain $X>0$ large enough.
Indeed, assume additionally $r\in L^1_{\loc} (\R)$, $r \neq 0$ a.e. on $\R$, and \eqref{e rinf}.
\cite[Proposition 2.5]{CL89} implies that the maximal operator $\A_r := -\frac 1r \frac {d^2}{dx^2}$
is definitizable. Now it is easy to see that statements (1) and (2) of Theorem \ref{t cp0} are valid with the same proof for $\A_r = -\frac 1r \frac {d^2}{dx^2}$.
\end{rem}

\section{Discussion \label{s disc}}

From another point of view algebraic multiplicities of eigenvalues of definitizable operators was considered in \cite[Proposition II.2.1]{Lan82} and  \cite[Section 1.3]{CL89} in terms of definitizing polynomials.
For operators $\frac {\sgn x}{|r(x)|} (-\frac{d}{dx}p\frac{d}{dx} +q)$, Theorem \ref{t s p GC} solves the same problem in terms of Titchmarsh-Weyl m-coefficients. Combining both approaches, it is possible
to get quite precise results both on eigenvalues and on definitizing polynomials. Such analysis was done in \cite[Section 4.3]{KarKosM09}
for operators $(\sgn x) (-d^2/dx^2+q)$ with potentials $q \in L^1 (\R; (1+|x|)dx) $; in particular, the minimal
definitizing polynomial was described in terms of Titchmarsh-Weyl m-coefficients (recently, \cite{K78} was used in \cite{BKT09pr} to extended some results of \cite[Section 4.3]{KarKosM09} on a slightly more general class of potentials).

A. Kostenko and later B. \'Curgus informed the author that Theorem \ref{t s p} and \cite[Section 6.1]{KarKos08} are in disagreement with one of the statements of
\cite[p. 39, 1st paragraph]{CL89}.

Namely, \cite[Section 1.3]{CL89} is concerned with $J$-self-adjoint operators $A$ 
(in a Krein space $K =(H , [\cdot, \cdot]))$ 
such that the form $[A \cdot, \cdot ]$ has a finite number $k_A $ of negative squares. Such operators  are sometimes called quasi-$J$-nonnegative.
It is assumed also that $\rho (A) \neq \emptyset$.
Here, as before, $H$ is a Hilbert space, $J$ is a  fundamental symmetry, and 
$[\cdot, \cdot]:=(J \cdot,\cdot)_{H}$.

According to \cite[p. 39, 1st paragraph]{CL89} (the author changes slightly the appearance):
 \begin{description}
 \item[\textbf{(p1)}] an operator $A$ of the type mentioned above has a definitizing polynomial $p_A$ of the form $p_A (z) = z q_A (z) \overline{q_A (\overline{z})}$, where the polynomial $q_A$ can be chosen monic and of minimal degree.  Under these assumptions, $q_A$ is unique and its degree is less than or equal to $k_A$. 
  \item[\textbf{(p2)}] A real number $\lambda \neq 0$ is a zero of $q_A $ if and only if it is an eigenvalue of $A$ such that $\lambda [f, f] \leq 0$ for some corresponding eigenvector $f$.
 \item[\textbf{(p3)}]
$q_A(0) = 0 $ implies that $0$ is an eigenvalue of $A$ and one of corresponding Jordan chains is of length $\geq 2$.
      \end{description}
      
Note also that in the settings of \cite[p. 39, 1st paragraph]{CL89}, $q_A$ is of minimal degree, but it easy to see the definitizing polynomial $p_A$ may be not a definitizing polynomial of minimal degree.

From the author's point of view,   two following statements in \cite[Section 1.3]{CL89} are incorrect: assertion \textbf{(p3)} and the equality $\dim \L_0 = k_A+ k_A^0 $ in \cite[Proposition 1.5]{CL89}.
Statement \textbf{(p3)} was given 
as a simple consequence of considerations in \cite{Lan82}. The proof of \cite[Proposition 1.5]{CL89} has an unclear point, which is discussed below. As far as the author understand, all other results of \cite{CL89} (as well as results obtained in \cite[Section 4.3]{KarKosM09}) do not depend of these two statements.

Let us explain points of contradiction in more details.
In \cite[Section 6.1]{KarKos08}, the operator $A$, defined by
\begin{equation}\label{e VI_1_01}
A:=JL, \quad (Jf) (x) = (\sgn x) f(x), \quad (Ly)(x)=-y''(x)+6\frac{x^4-6|x|}{(|x|^3+3)^2}y(x),
\end{equation}
is considered in $L^2 (\R )$. The operators $L$ and $A$ are defined on their maximal domains, $L$ is self-adjoint, and $A$ is $J$-self-adjoint.
We will use here notations of Section \ref{ss SLModel}. In particular, $\lexp [\cdot]$ is the differential expression of the operator $L$.

We will need the following properties of the operator 
$A$.

\begin{prop}[cf. Section 6.1 in \cite{KarKos08}] \label{p Adec}
Let $A$ be the $J$-self-adjoint operator defined by (\ref{e VI_1_01}). Then:
\begin{description}
\item[(i)] $A$ is a quasi-$J$-nonnegative operator,
  with $k_A =1$, i.e., the sesquilinear form $[A \cdot, \cdot ]$ has one negative square; 
\item[(ii)] $\sigma(A) \subset \R$; 
\item[(iii)] $\sigma_p (A) = \{0\}$;
\item[(iv)] $0$ is a simple eigenvalue of $A$.
\end{description}
\end{prop}

Statements (ii) and (iv) of Proposition \ref{p Adec} were proved in \cite[Theorem 6.1]{KarKos08}. Statement (i) was given in \cite[Remark 6.3]{KarKos08} with a very shortened proof.
For the sake of completeness, we give below the 
proofs of statements (i) and (iii).

\begin{proof}
Statement \textbf{(i)} follows from \cite[Remark 1.2]{CL89} and the fact that   the negative part of the spectrum of the operator $L$ consists of one simple eigenvalue at $\lambda_0 = -1$.

Let us prove that $\sigma(L) \cap (-\infty,0) = \{ -1\}$.
Consider the operators $L_{0}^\pm$ associated with the differential expression $\lexp [\cdot] $ and the Neumann problems $y'(\pm 0)=0 $ on $\R_\pm$,
and let $m_{\Sl\pm}$ be the corresponding Titchmarsh-Weyl m-coefficients, see e.g. \cite[formula (2.7)]{KarKosM09}. 

It follows from \cite[Lemma 6.2]{KarKos08} that both the Titchmarsh-Weyl m-coefficient $m_{\Sl\pm}$ are equal to the function 
$m_0 $ defined by 
\begin{equation}\label{VI_I_02}
m_0(\lambda)=\frac{\lambda}{1+\lambda(-\lambda)^{1/2}}, \qquad
\lambda \in \C \setminus \left( \{-1\} \cup [0,+\infty) \right),
\end{equation}
where $z^{1/2}$ denotes the branch of the complex root with a cut along the negative semi-axis $\R_-$ such that $(-1+i0)^{1/2}=i$. 
 
$\lambda_0 =-1$ is a pole of $m_0$, and therefore, an eigenvalue of both the operators $L_0^+$, $L_0^-$, and in turn an eigenvalue of $L$. 

The support of the spectral measure of $m_0$ is equal to $\{-1\} \cup [0,+\infty)$, see part (ii) of the proof of \cite[Theorem 6.1]{KarKos08}.
It is easy to see from the
standard  definition of Titchmarsh-Weyl m-coefficients
(or from \cite[Proposition 2.1]{DM91}) that 
$\lambda \in (-\infty,-1) \cup (-1,0)$ belongs to the spectrum of $L$ if and only if $m_{\Sl+} (\lambda) + m_{\Sl-} (\lambda) = 0$. But \cite[formula (6.2)]{KarKos08} implies that
$m_{\Sl+} (\lambda) + m_{\Sl-} (\lambda) = 2 m_0 (\lambda) \neq 0$ for all $\lambda \in (-\infty,-1) \cup (-1,0)$. 

Thus, the eigenvalue $\lambda_0 =-1$ is the only point of the spectrum of $L$ in $(-\infty,0)$.
$\lambda_0 =-1$ is a simple eigenvalue since $\lexp [\cdot ]$ is in the limit point case at $\pm \infty$
(see \cite[Theorem 5.3]{Weid87}).

\textbf{(iii)} It is proved in \cite[Theorem 6.1]{KarKos08} that $0$ is an eigenvalue of $A$
and that $\sigma(A) \subset \R$. Here we have to show   that $A$ has no eigenvalues in
$\R \setminus \{ 0 \}$.

The proof of \cite[Theorem 6.1 (ii)]{KarKos08} states that the Titchmarsh-Weyl m-coefficient
$M_{\Sl+}$ is equal to $m_0$, that its spectral measure $\Sigma_{\Sl+}$ is absolutely continuous on intervals  $[0,X]$, $X>0$, and that for $t>0$ we have
\[
\Sigma_{\Sl+}^\prime (t):=\frac{t^{5/2}}{\pi(1+t^3)}.
\]
Combining this with Theorem \ref{t s p} (1), we see that $(0,+\infty) \subset \mathfrak{A}_0 (\Sigma_{\Sl+})$,
and therefore $\sigma_p (A) \cap (0,+\infty) = \emptyset$. Since the potential of $L$ is even,
 we see that $\sigma_p (A) \cap (-\infty,0) = \emptyset$. This concludes the proof.
 \end{proof}
 
\begin{rem}
Actually, $\sigma(A)=\R$. This follows from Proposition \ref{p e and d} (i) and 
the fact that $M_{\Sl\pm} (\cdot) = \pm m_{\Sl\pm} (\pm \cdot)$. 
The fact that $A$ has no eigenvalues in $\R \setminus \{ 0 \}$ can also be easily obtained from  \cite[Section 5.7]{Tit58} or from \cite[Problem IX.4]{CodLev58Rus}.
\end{rem}

Combining Proposition \ref{p Adec} with \textbf{(p1)}
and \textbf{(p2)}, we will show that $q_A (z) = z $,
and that $q_A (z) = z $ contradicts  \textbf{(p3)}.
Indeed, since $L$ is not nonnegative, the polynomial $z$ is not a definitizing polynomial of the operator $A$.  So $p_A (z) \not \equiv z$, and therefore $q_A $ is nontrivial. $q_A $ has the degree equal to $k_A =1$ due to \textbf{(p1)}.
Since the polynomial $q_A$ is of minimal degree,
 Proposition \ref{p Adec} (ii) implies that $q_A$ has no non-real zeros, see \cite[p.11, the second paragraph]{Lan82} or \cite[p. 38, the last paragraph]{CL89}. (Note also that in our case $p_A$ is a definitizing polynomial of minimal degree since $0$ is a critical point of $A$.)
By Proposition \ref{p Adec} (iii), $A$ has no eigenvalues in
$\R \setminus \{ 0 \}$. Therefore, statement \textbf{(p2)} implies that $q_A$ has no zeros in
$\R \setminus \{ 0 \}$. 
Summarizing, we see that $q_A (z) = z$ and $p_A (z) = z^3$. Proposition \ref{p Adec} (iv) states that $0$ is a simple eigenvalue. This fact contradicts \textbf{(p3)}.

The equality $\dim \L_0 = k_A+ k_A^0 $ from \cite[Proposition 1.5]{CL89} is not valid for the operator $A$ defined by
\eqref{e VI_1_01}. 

Namely, \cite[Proposition 1.5]{CL89} states that there exists an invariant under $A$ subspace $\L_0$ of dimension  $\dim \L_0 = k_A+ k_A^0$, where $k_A^0$ is the dimension of the isotropic part of the root subspace $\Iroot_0 (A)$ with respect to the sesquilinear form $[ A \cdot, \cdot]$. For the operator $A$ defined by
\eqref{e VI_1_01}, 
statements (i) and (iv) of Proposition \ref{p Adec} imply that $k_A =1$ and $k_A^0 =1 $, respectively. 
So $\L_0$ is a two-dimensional invariant
subspace of $A$. All the root subspaces of the restriction $A\uph
\L_0 $ are root subspaces of $A$, and therefore
Proposition \ref{p Adec} (iii)-(iv) implies 
$\dim \L_0 \leq 1$. This contradicts $\dim \L_0 = k_A+ k_A^0 = 2$.

\begin{rem}
From the author's point of view, the statement 'the
inner product $[A \cdot , \cdot ] $ has $k_A^{\prime \prime}$ negative squares on $\Iroot_0 (A)$' in the proof of \cite[Proposition 1.5]{CL89} is not valid for the operator $A$ defined by \eqref{e VI_1_01}, since in this case $\Iroot_0 (A) = \ker A$, but $k_A^{\prime \prime} =1$.
\end{rem}

\appendix

\renewcommand{\theequation}{A-\arabic{equation}}
  \setcounter{equation}{0}  

\section{Appendix: Boundary triplets for symmetric operators}\label{s a}

In this section we recall necessary definitions and facts from the theory of  boundary triplets and abstract
Weyl functions following \cite{Koch75,GG91,DM91,DM95}.

Let $H$, $\H$, $\H_1$, and $\H_2$ be complex Hilbert spaces. By $[\H_1, \H_2]$ we denote   
the set of bounded linear operators acting from the space $\H_1$ to the space $\H_2$ and defined on all the space $\H_1$. If $\H_1=\H_2$, we write 
$[\H_1]$ instead of $[\H_1, \H_1]$.

Let $S$ be a
closed densely defined symmetric operator in $H$ with equal deficiency
indices \ $n_+ (S) = n_- (S)=n $ \ (by definition,
$n_{\pm}(S):=\dim \mathfrak{N}_{\pm i}(S)$, where
$\mathfrak{N}_\lambda(S):=\ker(S^* -\lambda I)$).

\begin{defn}\label{dI.3}
A triplet $\Pi = \{\H, \Gamma_0,\Gamma_1\}$ consisting of an
auxiliary Hilbert space $\H$ and linear mappings $\Gamma_j:
\dom(S^*) \longrightarrow \H$, \ $(j=0,\ 1)$, \ is called a
\emph{boundary triplet for} $S^*$ if the following two conditions
are satisfied:
\begin{description}
\item[(i)] \ $(S^*f,g)_H-(f,S^*g)_H = (\Gamma _1 f,\Gamma
_0g)_\H -        (\Gamma_0 f,\Gamma _1 g)_\H$, \qquad  $f,\ g
\in \dom(S^*)$;
\item[(ii)] the linear mapping $ \Gamma =
\{\Gamma_0 f,\Gamma_1 f\} \ : \dom(S^*) \longrightarrow \H \oplus
\H $ is surjective.
\end{description}
\end{defn}

In the rest of this section we assume that the Hilbert space $H$ is separable.
Then the existence of a boundary triplet for $S^*$
is equivalent to $n_+ (S) = n_- (S) $.

The mappings $\Gamma_0$ and $\Gamma_1$ naturally induce two
extensions $S_0 $ and $S_1$ of $S$ given
by
\[ S_j := S^* \, \upharpoonright \, \dom (S_j), \qquad \dom (S_j)
= \ker \Gamma_j, \quad j=0,1.
\]
It turns out that $S_0 $ and $S_1$ are self-adjoint operators in
$H$,\ \ $S_j^*=S_j, \ j=0,1$.

The \emph{$\gamma$-field} of the operator $S$ corresponding to the
boundary triplet $\Pi$ is the operator function $\gamma (\cdot) :
\rho(S_0) \to [\H, \mathfrak{N}_\lambda(S)]$ defined
by $\gamma (\lambda) := (\Gamma_0 \upharpoonright
\mathfrak{N}_\lambda(S))^{-1}$. The function $\gamma $ is
well-defined and holomorphic on $\rho (S_0)$.
\begin{defn}[\cite{DM91,DM95}]\label{dI.6}
Let $\Pi=\{\H, \Gamma_0,\Gamma_1\}$ be a boundary triplet for the
operator $S^*$. The operator-valued function $M(\cdot) : \rho
(S_0) \to [\H]$ defined by
$$
M(\lambda) := \Gamma_1 \gamma(\lambda),\qquad \lambda\in\rho
(S_0),
$$
is called \emph{the Weyl function} of $S$ corresponding to the
boundary triplet $\Pi$.
\end{defn}
Note that the Weyl function $M$ is holomorphic on $\rho (S_0)$ and is an (operator-valued) $(R)$-function obeying $0\in \rho(\im(M(i)))$.

\quad \\

\textit{The author expresses his gratitude to Paul Binding, Branko \'Curgus, Aleksey Kostenko, and
Cornelis van der Mee for useful discussions. The author would like to thank the anonymous referees
for careful reading of the paper and for numerous  suggestions on improving it. The author would like to thank the organizers of the conference IWOTA 2008 for 
the hospitality of the College of William and Mary.
This work was partly supported by the PIMS Postdoctoral Fellowship at the University of Calgary.}

\quad\\
I. M. Karabash,
\\Department of Mathematics and Statistics, University of Calgary, 2500 University Drive NW Calgary T2N 1N4, Alberta, Canada \\
and \\
Department of PDE, Institute of Applied Mathematics and Mechanics, R. Luxemburg str. 74, Donetsk 83114, Ukraine\\
e-mail: karabashi@yahoo.com, karabashi@mail.ru, karabash@math.ucalgary.ca
\end{document}